\documentclass[12pt, oneside, psamsfonts]{amsart}

\newif\ifPDF
\ifx\pdfoutput\undefined\PDFfalse
\else \ifnum \pdfoutput > 0 \PDFtrue
        \else \PDFfalse
        \fi
\fi

\usepackage[centertags]{amsmath}
\usepackage{amsfonts}
\usepackage{mathrsfs}
\usepackage{textcomp}
\usepackage{amssymb}
\usepackage{amsthm}
\usepackage{newlfont}
\usepackage[all]{xy}

\ifPDF
  \usepackage[pdftex]{color, graphicx}
  \usepackage[pdftex, bookmarks, colorlinks]{hyperref}
  \hypersetup{colorlinks=false}

\else
  \usepackage{color}
  \usepackage[dvips]{graphicx}
  \usepackage[dvips]{hyperref}
\fi

\usepackage[scale=0.8]{geometry}

\usepackage[pagewise, mathlines, displaymath]{lineno}

\newtheorem{thm}{Theorem}[section]
\newtheorem{cor}[thm]{Corollary}
\newtheorem{lem}[thm]{Lemma}

\theoremstyle{definition}
\newtheorem{defn}[thm]{Definition}
\theoremstyle{remark}
\newtheorem{rem}[thm]{Remark}

\numberwithin{equation}{section}

\newcommand{\norm}[1]{\left\Vert#1\right\Vert}
\newcommand{\abs}[1]{\left\vert#1\right\vert}

\newcommand{\Real}{\mathbb R}
\newcommand{\Int}{\mathbb Z}
\newcommand{\Comp}{\mathbb C}
\newcommand{\Ratn}{\mathbb Q}
\newcommand{\eps}{\varepsilon}

\newcommand{\Kzero}{\textrm{K}_0}
\newcommand{\Kone}{\textrm{K}_1}
\newcommand{\tr}{\mathrm{T}}
\newcommand{\MC}[2]{\mathrm{M}_{#1}(\mathrm{C}(#2))}
\newcommand{\MCO}[2]{\mathrm{M}_{#1}(\mathrm{C}_0(#2))}
\newcommand{\Mat}[1]{\mathrm{M}_{#1}(\mathbb C)}
\newcommand{\aff}{\mathrm{Aff}}

\begin{document}


\title{The classification of simple separable unital locally ASH algebras}

\author{George A. Elliott}
\address{Department of Mathematics, University of Toronto, Toronto, Ontario, Canada~\ M5S 2E4.}
\email{elliott@math.toronto.edu}

\author{Guihua Gong}
\address{Department of Mathematics, University of Puerto Rico.}
\email{ghgong@gmail.com}

\author{Huaxin Lin}
\address{University of Oregon and East China Normal University.}
\email{hlin@uoregon.edu}

\author{Zhuang Niu}
\address{Department of Mathematics, University of Wyoming, Laramie, WY, 82071, USA.}
\email{zniu@uwyo.edu}

\date{\today}


\begin{abstract}
Let $A$ be a simple separable unital locally approximately subhomogeneous C*-algebra (locally ASH algebra). It is shown that $A\otimes Q$ can be tracially approximated by unital Elliott-Thomsen algebras with trivial $\Kone$-group, where $Q$ is the universal UHF algebra. In particular, it follows that $A$ is classifiable by the Elliott invariant if $A$ is Jiang-Su stable.
\end{abstract}

\maketitle

\setcounter{tocdepth}{1}

\section{Introduction}
Recently, several major steps have been taken in the classification of what might be called ``well behaved" separable amenable simple unital C*-algebras. The phenomenon of well behavedness itself was explicitly noticed only relatively recently, by Toms and Winter (see \cite{E-Toms-BAMS}) who conjectured that within this class of C*-algebras several properties were equivalent, and that the algebras in this robust subclass (and only these) were the ones that could be classified by means of what might be called the naive Elliott invariants---the ordered $\Kzero$-group together with the class of the unit, the simplex of tracial states paired naturally with it, and the $\Kone$-group. (Other invariants, such as the Cuntz semigroup, might then be helpful for more general classes of amenable C*-algebras.) Breakthroughs in the understanding of the robustness of this class were made by Matui and Sato in \cite{Matui-Sato-CP} and \cite{Matui-Sato-DR}.

Perhaps the most striking development in the direction of actually proving isomorphism has been the technique---sometimes referred to as the Winter deformation technique---introduced by Winter in \cite{Winter-Z} (with refinements by Lin in \cite{Lin-App} and Lin and Niu in \cite{L-N}), through which a class of (separable, amenable, simple unital) C*-algebras which are well behaved in the sense of absorbing tensorially the Jiang-Su algebra $\mathcal Z$, and are also known to satisfy the UCT, can be classified in terms of the (naive) Elliott invariant if this is true for the subclass of algebras absorbing the universal Glimm UHF C*-algebra $Q$.

Using this, Gong, Lin, and Niu, in \cite{GLN-TAS}---following on earlier work in this direction (see e.g. \cite{LinTAI}, \cite{Niu-TAS-I},  \cite{Niu-TAS-II}, \cite{Niu-TAS-III}, \cite{Lin-Asy}, \cite{Winter-Z}, \cite{Lin-App}, \cite{L-N}, \cite{Lnclasn})---have achieved a classification of finite algebras in this well behaved class which is close to being definitive---it is simple to describe and it exhausts completely the possible values of the invariant. (The complementary class, the infinite algebras in the well behaved class under consideration, were dealt with some time ago by Kirchberg and Phillips---\cite{Kirch-infty}, \cite{KP0}, \cite{Ph1}.)

Unfortunately, while it is believed that all well behaved finite separable amenable simple unital C*-algebras may be ASH algebras (inductive limits of subalgebras of full matrix algebras over commutative C*-algebras)---and, indeed, that the algebra need not be assumed to be well behaved if in addition matrix algebras over it are also finite in the Murray-von Neumann sense (i.e., if the algebra is stably finite)---the class considered by Gong, Lin, and Niu does not on the face of it include the class of all well behaved---``Jiang-Su stable"---simple unital ASH algebras.

Using the recent result of Santiago, Tikuisis, and the present authors (\cite{ENST-ASH}) that any Jiang-Su stable simple unital ASH algebra has finite nuclear dimension (one of the Toms-Winter well behavedness properties---the important concept of nuclear dimension was introduced by Winter and Zacharias in \cite{WZ-ndim})---and also using the notion of non-commutative cell complex introduced in \cite{ENST-ASH} in the proof of this result---, the present note shows that indeed such an algebra (Jiang-Su stable simple unital ASH) belongs to the class dealt with by Gong, Lin, and Niu. (Even if the ASH algebra has slow dimension growth, so that by \cite{Toms-SDG} and \cite{Winter-Z-stable-02} it is indeed Jiang-Su stable, it is not known to belong to the class studied by Gong, Lin, and Niu---the class of rationally tracially approximately point--line algebras---see below.)

\subsection*{Acknowledgements}
The research of G.~A.~E.~was supported by a Natural Sciences
and Engineering Research Council of Canada (NSERC) Discovery Grant, the research of G.~G.~was supported by an NSF Grant, the research of H.~L.~was supported by an NSF Grant, and the research of Z.~N.~was supported by a Simons Foundation Collaboration Grant.
Z.~N.~also thanks Aaron Tikuisis for his comments on the early version of the paper.

\section{Noncommutative cell complexes}

\begin{defn}[Definition 2.1 of \cite{ENST-ASH}]
The class of noncommutative cell complexes (NCCC) is the smallest class $\mathcal C$ of C*-algebras such that
\par (1) every finite dimensional algebra is in $\mathcal C$; and
\par (2) if $B\in \mathcal C$, $n\in\mathbb N$, $\varphi: B\to \MC{k}{S^{n-1}}$ is a unital homomorphism, and $A$ is given by the pullback diagram
\begin{displaymath}
\xymatrix{
A\ar[r] \ar[d] & \MC{k}{D^n} \ar[d]^{f\to f|_{S^{n-1}}}\\
B \ar[r]_-{\varphi} & \MC{k}{S^{n-1}},
}
\end{displaymath}
then $A\in\mathcal C$.
\end{defn}

The reason we consider NCCCs is as follows: 
\begin{thm}[Theorem 2.15 of \cite{ENST-ASH}]
Let $A$ be a unital subhomogeneous C*-algebra. Then $A$ can be locally approximated by sub-C*-algebras which are NCCCs.
\end{thm}

\begin{defn}
Let $A$ be an NCCC, and fix an NC cell complex decomposition of $A$ with length $l$ (in the sense that $A$ is built from a finite dimensional C*-algebra $A_0$ by attaching $l$ noncommutative cells). Assume that $A_0=\Mat{s_1}\oplus\cdots\oplus\Mat{s_{r}}$ for some natural number $r$,  and denote the attached cell at the $i$-th step by $\MC{k_i}{D^{n_i}}$. 

Consider the spaces
$$\underbrace{\{\mathrm{pt}\}, ..., \{\mathrm{pt}\}}_r, D^{n_1}, ..., D^{n_l},$$
and denote them by
$X_1, X_2, ..., X_m$, where $m=r+l$. Denote the matrix orders of the corresponding C*-algebras by
$$d_1, ..., d_r, d_{r+1}, ..., d_m.$$
Then there is a standard embedding
$$ \Pi: A\to \bigoplus_{i=1}^m \MC{d_i}{X_i}.$$ 
Denote by $\Pi_i: A\to \MC{d_i}{X_i}$, $i=1, ..., m$, the projection of $\Pi$ onto the $i$-th direct summand.
\end{defn}

\begin{lem}\label{decp-tr-nccc}
Let $A= B\oplus_{\MC{k}{S^{n-1}}} \MC{k}{D^n}$ be an NCCC, and let $\tau\in \tr(A)$. Then there is a decomposition
$$\tau(f_B, f_D)=\alpha \tau_B(f_B)+\beta \int_{D^n\setminus S^{n-1}} \mathrm{tr}(f_D(x)) \mathrm{d}\mu(x),\quad (f_B, f_D)\in A,$$
where $\tau_D\in\tr(B)$, $\mu$ is a probability measure on $D^n\setminus S^{n-1}$, $\mathrm{tr}$ is the standard trace of $\Mat{k}$, and $\alpha, \beta\in[0, 1]$ with $\alpha+\beta=1$. Moreover, $\alpha$ and $\beta$ are unique and $\tau_B$ (or $\mu$) is unique if $\alpha$ (or $\beta$) is not zero.
\end{lem}
\begin{proof}
The uniqueness part of the lemma is clear. Let us show the existence part.
 
Consider the restriction of $\tau$ to $I:=\MCO{k}{D^n\setminus S^{n-1}}\subseteq A$. Then it is a trace with norm at most one, and thus there is $\beta\in[0, 1]$ and a probability measure $\mu$ on $D^n\setminus S^{n-1}$ such that
$$\tau((0, g))=\beta\int_{D^n\setminus S^{n-1}} \mathrm{tr}(g(x)) \mathrm{d}\mu(x),\quad g\in  \MCO{k}{D^n\setminus S^{n-1}}.$$

Define a linear map $\tilde{\tau}: A\to \Comp$ by
$$\tilde{\tau}((f, g))=\tau(f, g)-\beta\int_{D^n\setminus S^{n-1}} \mathrm{tr}(g(x)) \mathrm{d}\mu(x),\quad (f, g)\in A.$$
For each $g\in\MC{k}{D^n}$ and any $\eta\in(0, 1)$, define  
$$\chi_\eta: [0, 1]\ni x\mapsto \left\{
\begin{array}{ll}
1, & x\in[1-\eta/2, 1], \\
\textrm{linear}, & x\in[1-\eta, 1-\eta/2],\\
0, & x\in[0, 1-\eta];
\end{array}
\right.
$$
and
$$g_\eta(x)=g(x)\chi_\eta(\norm{x}).$$
Then a direct calculation shows that
\begin{equation}\label{cut-inv}
\tilde{\tau}(f, g)=\tilde{\tau}(f, g_\eta),\quad (f, g)\in A,\ \eta\in (0, 1).
\end{equation}

It is clear that $\tilde{\tau}$ is self-adjoint; let us show that it is positive. Let $(f, g)\in A$ be positive. Define
$$\delta=\inf\{\tau(f, g_\eta): \eta\in(0, 1)\}.$$

If $\delta=0$, let us show that $\tilde{\tau}(f, g)=0$. Indeed, in this case, one has
$$\tau((f, g))=\tau(f, g_\eta)+\tau(0, g-g_\eta)$$
and hence
\begin{equation}\label{approx-tr-sup}
\tau((f, g))=\sup\{\tau(0, g-g_\eta): \eta\in(0, 1)\}=\sup\{\beta\int_{D^n\setminus S^{n-1}} \mathrm{tr}_x(g-g_\eta)\mathrm{d}\mu(x): \eta\in(0, 1)\}.
\end{equation}
Note that for any $\eta\in(0, 1)$.
\begin{eqnarray*}
\tilde{\tau}((f, g)) & = & \tau(f, g)-\beta\int_{D^n\setminus S^{n-1}} \mathrm{tr}(g(x)) \mathrm{d}\mu(x) \\
& = & \tau(f, g)-\beta\int_{D^n\setminus S^{n-1}} \mathrm{tr}_x(g-g_\eta) \mathrm{d}\mu(x) \\
&&+\beta\int_{D^n\setminus S^{n-1}} \mathrm{tr}(g_\eta(x)) \mathrm{d}\mu(x),
\end{eqnarray*}
and since $\mu$ is a probability measure, the integral $\beta\int_{D^n\setminus S^{n-1}} \mathrm{tr}(g_\eta(x)) \mathrm{d}\mu(x)$ is arbitrarily small if $\eta$ is small enough. By \eqref{approx-tr-sup}, one has that $\tilde{\tau}((f, g))=0$.

If $\delta>0$, since $\mu$ is a probability measure, there is $\eta\in(0, 1)$ such that 
$$\beta\int_{D^n\setminus S^{n-1}} \mathrm{tr}(g_\eta(x))\mathrm{d}\mu(x)<\delta/2,$$
and therefore
$$\tilde{\tau}((f, g))=\tilde{\tau}((f, g_\eta))= \tau(f, g_\eta)-\beta\int_{D^n\setminus S^{n-1}} \mathrm{tr}(g_\eta(x)) \mathrm{d}\mu(x)\geq-\delta/2= \delta/2>0.$$

Therefore, one always has $\tilde{\tau}((f, g))\geq 0$, and so $\tilde{\tau}$ is a positive linear functional. Therefore $\tilde{\tau}$ is a (positive) trace of $A$. Note that $\tilde{\tau}(I)=0$, and therefore $\tilde{\tau}$ factors through $A/I\cong B$, and hence in fact is a trace of $B$. Therefore, there are $\alpha\in[0, 1]$ and $\tau_B\in\tr(B)$ such that $$ \tau(f, g)-\beta\int_{D^n\setminus S^{n-1}} \mathrm{tr}(g(x)) \mathrm{d}\mu(x)=\tilde{\tau}(f, g)=\alpha\tau_B(f),\quad  (a, b)\in A,$$ 
as desired.
\end{proof}

\begin{cor}\label{tr-measure}
Let $A$ be an NCCC with a given decomposition with length $l$. 
Then any trace $\tau$ of $A$ has a decomposition
$$\tau=\alpha_1\tau_1+\alpha_1\mu_1+\cdots+\alpha_m\mu_m,$$
where $m=\mathrm{rank}(\Kzero(A_0))+l$, $\mu_i$ is a probability measure on $D^{n_i}\setminus S^{n_i-1}$ if $X_i=D^{n_i}$, and $\mu_i$ is the Dirac measure if $X_i$ consists of a point, $\alpha_i\in[0, 1]$ and $\alpha_1+\alpha_2+\cdots+\alpha_m=1$. Moreover, the coefficients $\alpha_i$ are unique.
\end{cor}

\begin{defn}
Let $A$ be an NCCC with a given decomposition, and let $\tau\in\tr(A)$. 
Referring to Corollary \ref{tr-measure}, define $$\alpha_i(\tau)=\alpha_i.$$
\end{defn}

\begin{lem}\label{fg-K}
Let $A$ be a noncommutative cell complex (NCCC). Then the K-groups  of $A$ are finitely generated (as abelian groups).
\end{lem}
\begin{proof}
The statement is true if $A$ is finite dimensional. Assume the statement is true for noncommutative complexes with length at most $l$.

Let $A$ be a noncommutative complex with length $l+1$. Write $$A=B\oplus_{\MC{k}{S^{n-1}}} \MC{k}{D^{n}},$$
where $B$ is a noncommutative complex with length $l$. Then there is a short exact sequence
\begin{displaymath}
\xymatrix{
0 \ar[r] & \MCO{k}{\Real^{n}} \ar[r] & A \ar[r] & B \ar[r] & 0, 
}
\end{displaymath}
and the corresponding six-term exact sequence is
\begin{displaymath}
\xymatrix{
\Kzero(\mathrm{C}_0(\Real^{n}))\ar[r] & \Kzero(A)\ar[r] & \Kzero(B) \ar[d] \\
\Kone(B) \ar[u] & \Kone(A) \ar[l] & \Kone(\mathrm{C}_0(\Real^{n})). \ar[l]
}
\end{displaymath}
If $n$ is odd, one has
\begin{displaymath}
\xymatrix{
0 \ar[r] & \Kzero(A) \ar[r] & \Kzero(B) \ar[r] & \cdots ,
}
\end{displaymath}
and
\begin{displaymath}
\xymatrix{
0\ar[r] &\Int/m\Int \ar[r] & \Kone(A) \ar[r] & \Kone(B) \ar[r] & 0,
}
\end{displaymath}
for some positive integer $m$.
By the inductive hypothesis, the groups $\Kzero(B)$ and $\Kone(B)$ are finitely generated, and therefore the groups $\Kzero(A)$ and $\Kone(A)$ are finitely generated.

If $n$ is even, a similar argument shows that $\Kzero(A)$ and $\Kone(A)$ are finitely generated. Therefore, the K-groups of $A$ are always finitely generated. Hence by induction, the statement holds for all noncommutative cell complexes.
\end{proof}

In general, the positive cone of $\Kzero(A)$ might not be finitely generated; for instance, the positive cone of $\Kzero(\mathrm{C}(S^2))$ is
$$\{(m, n)\in\Int^2: m>0\}\cup\{(0, 0)\},$$
which is not finitely generated. (On the other hand, consider the image $$G:=\rho(\Kzero(A))\subseteq \aff(\tr(A)),$$ with respect to the canonical map $\rho$, with the induced order from $\aff(\tr(A))$ (i.e., an element $g\in G$ is positive if and only if $g$ is positive in $\aff(\tr(A))$), is isomorphic to $\Int$ and so the positive cone of $G$ is finitely generated.)

The following lemma was stated and proved in \cite{GLN-TAS} for the $\Kzero$-group of an Elliott-Thomsen algebra. The argument in fact shows the following (for the reader's convenience, we include the proof). (In fact the ordered groups arising are the same.)
\begin{lem}[Theorem 3.14 of \cite{GLN-TAS}]\label{embedding-fg}
Consider $(\Int^l, (\Int^l)^+)$ with the standard (direct sum) order.  Let $G$ be a subgroup of $\Int^m$, and put $G^+=G\cap (\Int^m)^+$. Then the positive cone $G^+$ is finitely generated (as a semigroup).
\end{lem}

\begin{proof}
Let us first show that $G^+\setminus\{0\}$ has only finitely many minimal elements. Suppose, otherwise, that $\{q_n\}$ is an infinite set of minimal elements  of $G_+\setminus \{0\}.$
Write $$q_n=(m(1,n), m(2,n), ..., m(j, n))\in \Int^m_+,$$ where $m(i,n)$ are positive integers (including zero),
$i=1,2,...,m$ and $n=1, 2,....$ If there is an integer $M\ge 1$ such that
$m(i,n)\le M$ for all $i$ and $n,$ then $\{q_n\}$ is a finite set. So we may assume
that $\{m(i,n)\}$ is unbounded for some $1\le i\le m.$
Passing to a subsequence of $\{n_k\}$ such that
$\lim_{k\to\infty} m(i,n_k)=+\infty,$ we may assume
that  $\lim_{n \to\infty}m(i,n)=+\infty.$ We may assume that, for some $j,$ $\{m(j,n)\}$ is bounded.
Otherwise, by passing to a subsequence,  we may assume that $\lim_{n\to\infty}m(i,n)=+\infty$ for
all $i\in \{1,2,...,m\}.$
Therefore  $\lim_{n\to\infty}m(i,n)-m(i,1)=+\infty.$
It follows that, for some $n\ge 1,$
$m(i,n)>m(i,1)$ for all $i\in\{1,2,...,m\}.$ Therefore, $q_n> q_1$, which contradicts the fact
that $q_n$ is minimal.
By passing to a subsequence, we may write $\{1,2,...,m\}=N\sqcup B$ such that
$\lim_{n\to\infty} m(i,n)=+\infty$ if $i\in N$ and
$\{m(i,n)\}$ is bounded if $i\in B.$ Therefore, $\{m(j,n)\}$ has only finitely many different values if $j\in B.$
Thus, by passing to a subsequence  again, we may assume that
$m(j,n)=m(j,1)$ if $j\in B.$ Therefore, for some $n>1,$
$m(i,n)>m(i,1)$ for all $n$ if $i\in N$, and $m(j,n)=m(j,1)$ for all $n$ if $j\in B.$
It follows that $q_n\ge q_1.$ This is impossible, since $q_n$ is minimal.
This shows that $G^+$ has only finitely many minimal elements.

To show that $G^+$ is generated by these minimal elements,
fix an element $q\in G^+\setminus \{0\}.$   It suffices to show that $q$ is a finite sum
of minimal elements in $G^+.$ If $q$ is not minimal, consider
the set of all elements in $G^+\setminus\{0\}$ which are (strictly) smaller than $q.$
This set is finite.  Choose one which is minimal among them, say $p_1.$ Then $p_1$ is minimal  element
in $G^+\setminus \{0\},$ as otherwise there is one smaller than $p_1.$  Since $q$ is not minimal, $q\not=p_1.$
Consider $q-p_1\in G^+\setminus\{0\}.$ If $q-p_1$ is minimal, then
$q=p_1+(q-p_1).$ Otherwise, we repeat the same argument to obtain a minimal element
$p_2\le q-p_1.$ If $q-p_1-p_2$ is minimal, then $q=p_1+p_2+(q-p_1-p_2).$
Otherwise we repeat the same argument. This process is finite. Therefore $q$ is a finite sum of
minimal elements in $G^+\setminus \{0\}.$
\end{proof}

With Lemma \ref{embedding-fg}, one has

\begin{lem}\label{fg-quotient}
Let $A$ be an NCCC. Then the ordered group $$(\rho_A(\Kzero(A)), \rho_A(\Kzero(A))\cap \aff^+(\tr(A)))$$ is finitely generated (as an ordered group). (In other words, the positive cone is finitely generated as a semigroup.)
\end{lem}
\begin{proof}
With the fixed NC cell complex decomposition of $A$, consider the standard embedding
$$\Pi: A\to \bigoplus_{i=1}^m \MC{d_i}{X_i}.$$

Define
$$\rho: \Kzero(A)\ni [p] \mapsto (\mathrm{rank}(\Pi_1(p)), ..., \mathrm{rank}(\Pi_m(p))) \in \Int^m.$$
Clearly, the map $\rho$ is positive.

Define
$$G=\rho(\Kzero(A)) \quad\mathrm{and}\quad G^+=\rho(\Kzero(A))\cap (\Int^m)^+.$$
It follows from Lemma \ref{embedding-fg} that the cone $(G, G^+)$ is a finitely generated ordered group. In order to prove the lemma, one only has to show that there is an isomorphism
\begin{equation*}
(G, G^+) \cong (\rho_A(\Kzero(A)), \rho_A(\Kzero(A))\cap\aff^+(\tr(A))).
\end{equation*}

Define
$$\theta: \rho(\Kzero(A))\ni (x_1, x_2, ..., x_m) \mapsto (\tau\mapsto \sum_{i=1}^m\frac{x_i}{d_i}\alpha_i(\tau))\in \aff(\tr(A)).$$
Then $\theta$ is injective. Note that for any $\tau\in\tr(A)$ with the decomposition $$\tau=\alpha_1\mu_0+\alpha_2\mu_2+\cdots+\alpha_m\mu_m,$$
one has
\begin{eqnarray*}
\rho_A(p)(\tau) & = & \sum_{i=1}^m \alpha_{i}\int_{X_i} \mathrm{tr}_i(\Pi_i(p)(x))\mathrm{d}\mu_{i}(x) \\ 
& = & \sum_{i=1}^m \alpha_{i} \int_{X_i} \frac{\mathrm{rank}(\Pi_i(p))}{d_i} \mathrm{d}\mu_{i, j} \\
& = & \sum_{i=1}^m \alpha_{i} \frac{\mathrm{rank}(\Pi_i(p))}{d_i}=\theta(\rho(p)),
\end{eqnarray*}
and therefore
$$\theta(\rho(\Kzero(A)))=\theta(G)=\rho_A(\Kzero(A)).$$

It is clear that $\theta(G^+)\subseteq \rho_A(\Kzero(A))\cap\aff^+(\tr(A))$. Moreover, if $\theta(x_1, ..., x_m)\in \aff^+(\tr(A))$, then the affine map $\tau\mapsto \sum_{i=1}^m\frac{x_i}{d_i}\alpha_i(\tau)$ is positive, and hence each $x_i$ must be positive; that is, $\theta$ induces an order isomorphism between $(G, G^+)$ and $\rho_A(\Kzero(A)), \rho(\Kzero(A))\cap\aff^+(\tr(A))$, as desired.
\end{proof}

\begin{lem}\label{approx-finite-tr}
Let $A$ be an NCCC. Then, for any finite set $\mathcal F\subseteq \aff(\tr(A))$ and any $\eps>0$, there are positive continuous affine maps $\theta_1: \aff(\tr(A))\to\Real^s$ and $\theta_2: \Real^s\to \aff(\tr(A))$ for some $s\in\mathbb N$ such that
$$\norm{\theta_2\circ\theta_1(f)-f}_\infty<\eps,\quad f\in\mathcal F.$$
\end{lem}
\begin{proof}
The statement clearly holds for $A$ a finite dimensional C*-algebra. Assume that $$A=B\oplus_{\MC{k}{S^{n-1}}}\MC{k}{D^n},$$
and suppose that the statement holds for $B$.

Let $(\mathcal F, \eps)$ be given. Note that $\aff(\tr(A))$ is the pullback of $\aff(\tr(B))$ and $\mathrm{C}_{\Real}(D^n)$ in the same manner as $A$.  For each $f\in\mathcal F$, write it as $f=(f_B, f_D)$, where $f_B\in \aff(\tr(B))$ and $f_D\in\mathrm{C}_{\Real}(D^n)$. Denote by $\mathcal F_B$ the set of $f_B$'s. 

Since $B$ satisfies the statement, there are continuous positive affine maps $\theta_{B, 1}: \aff(\tr(B))\to \Real^{s_B}$ and $\theta_{B, 2}: \Real^{S_B}\to \aff(\tr(B))$ such that
\begin{equation}\label{small-B}
\norm{\theta_{B, 2}\circ\theta_{B, 1}(f_B)-f_B}_\infty<\eps/3,\quad f\in\mathcal F.
\end{equation}

Choose $\gamma\in(0, 1)$ such that 
$$\norm{f_D(x)-f_D(y)}<\eps/3,\quad f\in\mathcal F,$$
provided $\mathrm{dist}(x, y)<\gamma$.

Define
$$g_\gamma: [0, 1]\ni x\mapsto \left\{
\begin{array}{ll}
1, & x\in[0, 1-\gamma], \\
\textrm{linear}, & x\in[1-\gamma, 1-\gamma/2],\\
0, & x\in[1-\gamma/2, 1];
\end{array}
\right.
$$
and consider the linear map
$$\chi_\gamma: \mathrm{C}(S^{n-1}) \ni f \mapsto (x\mapsto (1-g_\gamma(\norm{x})) \cdot f(\frac{x}{\norm{x}}))\in \mathrm{C}(D^n).$$
Then $\chi_\gamma$ is a positive affine map from $\aff(\tr(\MC{k}{S^{n-1}}))$ to $\aff(\tr(\MC{k}{D^n}))$. 
Note that
\begin{equation}\label{bd-cut}
\norm{f_Dg_\gamma-(f_D- \chi_\gamma (f_D|_{S^{n-1}}))}_\infty=\sup\{(1-g_\gamma(\norm{x}))(f_D(x)-f_D(\frac{x}{\norm{x}})): x\in D^n\}<\eps/3
\end{equation}
for any $f\in\mathcal F$.


Pick an open cover $\mathcal U$ of $D^n$ such that the variation of the function $f_Dg_\gamma$ on any open subset in $\mathcal U$ is at most $\eps/3$ for any $f\in\mathcal F$. Moreover, one requires that the diameter of each open set in $\mathcal U$ should be at most $\gamma/4$. Choose $\{\phi_{U}: U\in \mathcal U\}$ to be a partition of unity subordinated to $\mathcal U$, and fix $x_U\in U$ for each $U\in {\mathcal U}$.

Put $$\mathcal U_\gamma=\{U_i: U\cap S^{n-1}=\O\},$$ and write
$\mathcal U_\gamma=\{U_1, U_2, ..., U_{\abs{{\mathcal U_\gamma}}}\}.$

Since the diameter of each $U_i\in\mathcal U$ is at most $\gamma/4$, one has that if $U\notin\mathcal U_\gamma$, then $g_\gamma(x_U)=0$, and hence
\begin{eqnarray}\label{small-pert}
 \norm{f_Dg_\gamma-\sum_{U\in\mathcal U_\gamma}(f_Dg_\gamma)(x_U)\phi_U} 
&=&  \norm{f_Dg_\gamma-\sum_{U\in\mathcal U_\gamma}(f_Dg_\gamma)(x_U)\phi_U- \sum_{U\notin\mathcal U_\gamma}(f_Dg_\gamma)(x_U)\phi_U} \\
&=& \norm{f_Dg_\gamma-\sum_{U\in\mathcal U}(f_Dg_\gamma)(x_U)\phi_U}\leq \eps/3. \nonumber
\end{eqnarray}

Define
$$\theta_1: \aff(\tr(A))\ni (f, g)\mapsto \theta_{B, 1}(f) \oplus ((gg_\gamma)(x_{U_i})) \in \Real^{s_B}\oplus\Real^{\abs{\mathcal U_\gamma}},$$
and
\begin{eqnarray*}
\theta_2: \Real^{s_B}\oplus\Real^{\abs{\mathcal U_\gamma}}  & \ni &  ((\xi_1, ..., \xi_{s_B}), (\eta_1, ..., \eta_{\abs{\mathcal U_\gamma}})) \\ 
& \mapsto &   (\theta_{B, 2}((\xi_1, ..., \xi_{s_B})),  \chi_\gamma(\varphi(\theta_{B, 2}((\xi_1, ..., \xi_{s_B}))))+ \sum_{i=1}^{\abs{\mathcal U_\gamma}} \eta_i\phi_{U_i}) \in  \aff(\tr(A)).
\end{eqnarray*}

Then, a straightforward calculation shows that 
$$\theta_2\circ\theta_1((f_B, f_D))=(\theta_{B, 2}(\theta_{B, 1}(f_B)),  \chi_\gamma(\varphi(\theta_{B, 2}(\theta_{B, 1}(f_B))))+ \sum_{i=1}^{\abs{\mathcal U_\gamma}} (f_Dg_\gamma)(x_{U_i})\phi_{U_i})$$

By the inductive assumption, one has that for any $f\in\mathcal F$,
$$\norm{f_B-\theta_{B, 2}(\theta_{B, 1}(f_B))}<\eps/3<\eps,$$
and 
\begin{eqnarray*}
&& \norm{f_D- (\chi_\gamma(\varphi(\theta_{B, 2}(\theta_{B, 1}(f_B))))+ \sum_{i=1}^{\abs{\mathcal U_\gamma}} (f_Dg_\gamma)(x_{U_i})\phi_{U_i})}_\infty \\
& < & \norm{f_D- (\chi_\gamma(f_D|_{S^{n-1}})+ \sum_{i=1}^{\abs{\mathcal U_\gamma}} (f_Dg_\gamma)(x_{U_i})\phi_{U_i})}_\infty  +\eps/3\quad\textrm{(by \eqref{small-B})}\\
& = & \norm{(f_D- \chi_\gamma(f_D|_{S^{n-1}})- \sum_{i=1}^{\abs{\mathcal U_\gamma}} (f_Dg_\gamma)(x_{U_i})\phi_{U_i})}_\infty  +\eps/3 \\
& \leq  & \norm{g_\gamma f_D- \sum (f_Dg_\gamma)(x_p)\phi_p}_\infty  +2\eps/3\quad \textrm{(by \eqref{bd-cut})} \\
& \leq & \eps \quad \textrm{(by \eqref{small-pert})}.
\end{eqnarray*} 
Therefore
$$\norm{\theta_2\circ \theta_1(f)-f}<\eps,\quad f\in\mathcal F,$$
as desired.
\end{proof}

\begin{rem}
In fact, as shown in \cite{EN-K0-Z} (Lemma 2.6), the statement of Lemma \ref{approx-finite-tr} always holds if $\tr(A)$ is replaced by an arbitrary compact metrizable Choquet simplex.
\end{rem}

\section{An existence theorem}

Let $A$ be an NCCC. The main result in this section (Theorem \ref{est-AM}) is that any almost compatible pair $(\kappa, \gamma)$ from an NCCC to $Q$ can be lifted to an algebra homomorphism, where $\kappa: \Kzero(A)\to\Ratn$ and $\gamma: \aff(\tr(A))\to\aff(\tr(Q))\cong\Real$.

\begin{lem}\label{pert-sol}
Let $P$ be an $m\times n$ matrix with integer entries, and let $\xi\in \Real^m$ with each entry a positive number (including zero). Assume that each entry of $P\xi$ is rational. Then, for any $\eps>0$, there is $\zeta\in\Ratn^m$ with positive entries (including zero) such that
\begin{enumerate}
\item $P\xi=P\zeta$, and
\item $\norm{\xi-\zeta}_\infty<\eps.$
\end{enumerate} 
\end{lem}

\begin{proof}
By deleting the columns of $P$ corresponding to the $0$'s of $\xi$, one may assume that each entry of $\xi$ is strictly positive. 

Let us show that $$\ker P=\overline {\ker P\cap\Ratn^m}.$$ It is clear that $\ker P \supseteq \overline {\Ratn^m\cap \ker P}.$ Note that $P$ is rational, so that one can choose a basis of $\ker P$ (as a real vector space) consisting of rational vectors, from which it follows that
$$\mathrm{dim}_\Real(\ker P) \leq  \mathrm{dim}_{\Ratn} (\ker P\cap\Ratn^m),$$ and hence $\mathrm{dim}_\Real(\ker P) \leq  \mathrm{dim}_{\Real} (\overline{\ker P\cap\Ratn^m})$.  This forces $\ker P=\overline {\Ratn^m\cap \ker P}.$

Again by the rationality of $P$, there is a vector $\eta\in\Ratn^m$ such that
$$P\eta=P\xi,$$
and hence
$$P^{-1}(\{P\xi\})=P^{-1}(\{P\eta\})=\eta+\ker P=\eta+ \overline {\Ratn^m\cap \ker P}=\overline {\eta+\Ratn^m\cap \ker P}.$$
It is clear that
$$\xi\in \overline {\eta+\Ratn^m\cap \ker P}.$$
On noting that any entry of $\xi$ is strictly positive, and all vectors in $\eta+\Ratn^m\cap \ker P$ are rational, for any $\eps>0$, it follows that there is $\zeta\in \eta+\Ratn^m\cap \ker P$ such that 
$$\norm{\zeta-\xi}_\infty<\eps$$
and each entry of $\zeta$ is positive.
\end{proof}

\begin{lem}\label{est-exact}
Let $A$ be a unital subhomogeneous C*-algebra such that $\mathrm{Prim}_d(A)$ has finitely many connected components for each $d$. Let $(\mathcal F, \eps)$ be given. Then, for any compatible pair $(\kappa, \gamma)$ satisfying, where $\kappa: \Kzero(A)\to \Ratn=\Kzero(Q)$ and $\gamma: \aff(\tr(A))\to \Real=\aff(\tr(A))$, there is a homomorphism $\phi: A\to Q$ such that
\begin{enumerate}
\item $[\phi]_0=\kappa,$ and
\item $|\gamma(\tilde{f})(\mathrm{tr})-\mathrm{tr}(\phi(f))|<\eps,\quad f\in\mathcal F,$
\end{enumerate}
where $\mathrm{tr}$ is the canonical trace of $Q$. Moreover, $\phi$ can be chosen to have finite dimensional range.
\end{lem}

\begin{proof}
Without loss of generality, one may assume that $\mathcal F$ is inside the unit ball.
Since $Q$ has unique trace, it is enough to show that for any $\kappa: \Kzero(A) \to \Ratn$ and $\tau\in\tr(A)$ with
$$\tau(p)=\kappa(p),\quad p\in\Kzero(A),$$
there is a homomorphism $\phi: A\to Q$
such that $[\phi]_0=\kappa$ and
$$\abs{\tau(f)-\mathrm{tr}(\phi(f))}<\eps,\quad f\in\mathcal F.$$

By Corollary \ref{tr-measure}, there is a probability measure $\mu_{i}$ on each $X_{i}$, and $\alpha_{i}\in [0, 1]$ with 
$$\sum_{i=1}^m\alpha_{i}=1$$ such that
$$\tau(a)=\sum_{i=1}^m \alpha_{i}\int_{X_{i}} \mathrm{tr}_i(\Pi_i(a))\mathrm{d}\mu_{i},\quad a\in A,$$
where $\mathrm{tr}_i$ is the canonical trace of $\Mat{k_i}$.

Note that for each projection $p\in A\otimes\mathcal K$, one has
\begin{eqnarray}\label{eva-proj}
\tau(p) & = & \sum_{i=1}^m \alpha_{i}\int_{X_i} \frac{\mathrm{rank}(\Pi_i(p))}{d_i}\mathrm{d}\mu_{i} \nonumber \\
& = & \sum_{i=1}^m \alpha_{i}  \frac{\mathrm{rank}(\Pi_i(p))}{d_i}.
\end{eqnarray}

Since $\Kzero(A)$ is finitely generated, there are $p_1, p_2, ..., p_r\in \Kzero(A)$ which generate $\Kzero(A)$.
Consider the matrix
\begin{displaymath}
P=\left(
\begin{array}{cccc}
\mathrm{rank}(\Pi_1(p_1)) & \mathrm{rank}(\Pi_2(p_1)) & \cdots & \mathrm{rank}(\Pi_m(p_1)) \\
\mathrm{rank}(\Pi_1(p_2)) & \mathrm{rank}(\Pi_2(p_2))& \cdots & \mathrm{rank}(\Pi_m(p_2))\\
\vdots &\vdots  & \vdots & \vdots \\
\mathrm{rank}(\Pi_1(p_r)) & \mathrm{rank}(\Pi_2(p_r)) & \cdots & \mathrm{rank}(\Pi_m(p_r))\\
d_1 & d_2 & \cdots & d_m
\end{array}
\right),
\end{displaymath}
and the vector
$$\xi=(\frac{\alpha_1}{d_1}, \frac{\alpha_2}{d_2}, ..., \frac{\alpha_m}{d_m})^{\mathrm{T}}.$$
By \eqref{eva-proj}, one has
$$P\xi=(\tau(p_1), \tau(p_2), ..., \tau(p_r), 1)^{\mathrm{T}}=(\kappa(p_1), \kappa(p_2), ..., \kappa(p_r), 1)^{\mathrm{T}}\in\Ratn^{r+1}.$$
Then, by Lemma \ref{pert-sol}, there is a positive rational vector $\xi\in\Ratn^m$ such that
\begin{equation}\label{new-vector}
P\xi=P\zeta\quad\mathrm{and}\quad\norm{\xi-\eta}_\infty < \eps/2l(k_1+\cdots+k_l).
\end{equation}
Write $$\zeta=(\frac{\beta_1}{d_1}, \frac{\beta_2}{d_2}, ..., \frac{\beta_m}{d_m})^{\mathrm{T}}.$$
By \eqref{new-vector}, one has that 
\begin{equation}\label{small-coef}
\abs{\alpha_{i}-\beta_{i}}<\eps/2m,\quad 1\leq i\leq m.
\end{equation}

Since $P\xi=P\zeta$, one has 
$$\sum_{i=1}^m\beta_{i} \frac{\mathrm{rank}(\pi_x(p_j))}{d_i}=\sum_{i=1}^m\alpha_{i} \frac{\mathrm{rank}(\pi_x(p_j))}{d_i},\quad j=1, ..., r. $$
Since $p_1, p_2, ..., p_l$ generate $\Kzero(A)$, one has
\begin{equation}\label{unchange-proj}
\sum_{i=1}^m\beta_{i} \frac{\mathrm{rank}(\pi_x(p))}{d_i}=\sum_{i=1}^m\alpha_{i} \frac{\mathrm{rank}(\pi_x(p))}{d_i},\quad p\in \Kzero(A).
\end{equation}
It also follows from $P\xi=P\zeta$ that $\beta_1+\cdots+\beta_m=1$.

Consider $\tau'\in \tr(A)$ defined by
$${\tau'}(a)=\sum_{i=1}^m\beta_{i}\int_{X_i} \mathrm{tr}_i(\Pi_i(a))\mathrm{d}\mu_{i},\quad a\in A.$$
Then, by \eqref{unchange-proj} and \eqref{small-coef}, one has
$${\tau'}(p)=\tau(p),\quad p\in\Kzero(A),$$
and
$$\abs{\tau(f)-{\tau'}(f)}<\eps/2,\quad f\in\mathcal F.$$

Consider each probability measure $\mu_{i}$, and choose a discrete measure $\tilde{\mu}_{i}$ such that
$$\abs{\int_{X_{i}}\mathrm{tr}_i(\Pi_i(f))\mathrm{d}\mu_{i}-\int_{X_{i}}\mathrm{tr}_i(\Pi_i(f))\mathrm{d}\tilde{\mu}_{i}}<\eps/2,\quad f\in\mathcal F.$$

Write 
$$\tilde{\mu}_{i}=\frac{1}{l_{i}}\sum_{k=1}^{l_{i}}\delta_{x_{ik}},$$
for some $x_{ik}\in X_{i}$, where $\delta_{x_{ik}}$ is the Dirac measure concentrated at $x_{ik}$. Without loss of generality, one may assume that all $l_{i}$ are the same and equal to $l$ for some $l$.

Define 
$$\tilde{\tau}(a)=\sum_{i=1}^m \beta_{i}\int_{X_i} \mathrm{tr}_i(\Pi_i(a))\mathrm{d}\tilde{\mu}_{i},\quad a\in A.$$
One then has 
$$\tilde{\tau}(p)=\tau(p),\quad p\in\Kzero(A),$$
and
$$\abs{\tau(f)-\tilde{\tau}(f)}<\eps,\quad f\in\mathcal F.$$
Write $\beta_{i}=q_{i}/q$ for natural numbers $q_{i}\leq q$. Since $\sum\beta_{i}=1$, one has that $$\sum q_{i}=q.$$

Define a unital homomorphism by
$$\phi: A\ni a\mapsto \bigoplus_{i=1}^m\bigoplus_{k=1}^{l} (\underbrace{\pi_{x_{ik}}(a)\oplus\cdots\oplus \pi_{x_{ik}}(a)}_{d_1 \cdots d_{i-1}d_{i+1}\cdots d_mq_{i}})\in \Mat{N},$$
where $N=d_1 \cdots d_m lq.$
Then
\begin{eqnarray*}
\mathrm{tr}(\phi(a)) 
& = & \frac{\sum_{i=1}^m  \sum_{k=1}^l(d_1 \cdots d_{i-1}d_{i+1}\cdots d_mq_i)\mathrm{Tr}(\pi_{x_{ik}}(a))}{d_1\cdots d_n ql}\\
& = & \sum_{i=1}^m \frac{d_1 \cdots d_{i-1}d_{i+1}\cdots d_md_{i}}{d_1\cdots d_n q}q_i\frac{1}{l}\sum_{k=1}^l\mathrm{tr}_i(\pi_{x_{ik}}(a))\\
& = & \sum_{i=1}^m \frac{q_{i}}{q}(\frac{1}{l} \sum_{k=1}^l\mathrm{tr}_i(\pi_{x_{ik}}(a))) = \tilde{\tau}(a).
\end{eqnarray*}
Let $\iota: \Mat{N}\to Q$ be an unital embedding. Then the homomorphism $\iota\circ\phi$ satisfies the condition of the lemma.
\end{proof}

\begin{lem}\label{lbd-test}
Let $A$ be a NCCC. For any finite set $\mathcal F\subseteq A$ and any $\eps>0$, there is a finite set $\mathcal H\subseteq A^+$ such that for any $\sigma>0$, there is $\delta>0$ such that if $\tau\in\tr(A)$ satisfies
$$\tau(h)>\sigma,\quad h\in\mathcal H,$$
there exists $\tau'\in\tr(A)$ such that 
\begin{enumerate}
\item $\tau'(p)=\tau(p),$ $p\in\Kzero(A)$,
\item $\abs{\tau'(f)-\tau(f)}<\eps$, $f\in\mathcal F$, and
\item $\alpha_{i}(\tau')>\delta$, all $i$.
\end{enumerate}
Moreover, one may require that $\delta$ depends only on $\sigma$.
\end{lem}

\begin{proof}
If $A=\Mat{n_1}\oplus\cdots\oplus\Mat{n_k}.$ Then $\mathcal H=\{1_{\Mat{n_i}}: 1\leq i\leq k\}$ satisfies the lemma with $\delta=\sigma$ for any given $\sigma$ (in fact, one has that $\tau'=\tau$ in this case).

Let $$A=B\oplus_{\MC{k}{S^{n-1}}} \MC{k}{D^{n}},$$
where $B$ is NCCC, and assume that the lemma holds for $B$.

Let $\mathcal F\subseteq A$, $\eps>0$ be given. For each $f\in A$, write $f=(f_B, f_D)$ for $f_B\in B$ and $f_D\in \MC{k}{D^n}$. Since $B$ is assumed to satisfy the lemma, there is $\mathcal H_B$ such that for any $\sigma$, there is $\delta_B(\sigma)>0$ such that
if $\tau_B\in\tr(B)$ satisfies
$$\tau_B(h)>\sigma,\quad h\in\mathcal H_B,$$
there is $\tau_B'\in\tr(A)$ such that
$$\tau_B'(p)=\tau_B(p),\quad p\in\Kzero(B),$$
$$\abs{\tau_B'(f_B)-\tau_B(f_B)}<\eps/2,\quad f\in \mathcal F,$$ and
$$\alpha_{i}(\tau'_B)>\delta_B(\sigma),\quad \textrm{all $i$}.$$

Choose $\gamma>0$ such that 
$$\norm{f_D(x)-f_D(y)}<\eps/2,\quad f\in\mathcal F,$$
if $\mathrm{dist}(x, y)<\gamma$.

Define 
$$g_\gamma: [0, 1]\ni x\mapsto \left\{
\begin{array}{ll}
1, & x\in[1-\gamma/2, 1], \\
\textrm{linear}, & x\in[1-\gamma, 1-\gamma/2],\\
0, & x\in[0, 1-\gamma/2];
\end{array}
\right.
$$
similarly, define $g_{\gamma/2}$ and $g_{2\gamma}$.

For each $h\in\mathcal H_B$ define $h_D\in \MC{k}{D^n}$ by
$$
h_D(x)=\left\{
\begin{array}{ll}
0, & \norm{x}\in[0, 1-\gamma/2],\\
\frac{4\norm{x}+2\gamma-4}{\gamma} \varphi(h)(\frac{x}{\norm{x}})  , &  \norm{x}\in[1-\gamma/2, 1-\gamma/4], \\
  \varphi(h)(\frac{x}{\norm{x}}), &  \norm{x}\in[1-\gamma/4, 1],  
\end{array}
\right.
$$
where $\varphi: B\to\MC{k}{S^{n-1}}$ is the gluing map.

For each $h\in\mathcal H_B$, define 
$$\tilde{h}=(h, h_D)\in A=B\oplus_{\MC{k}{S^{n-1}}}\MC{k}{D^n}.$$
Also define $\tilde{g}_{2\gamma}=(0, g_{2\gamma} 1_{\Mat{k}}), \tilde{g}_{\gamma/2}=(0, g_{\gamma/2} 1_{\Mat{k}})\in A$.

Then 
$$\mathcal H=\{\tilde{h}, \tilde{g}_{2\gamma}, 1-\tilde{g}_{\gamma/2}: h\in \mathcal H_B\}$$
satisfies the condition of the lemma. 

Indeed, for any $\sigma>0$, set
$$\delta=\sigma\delta_B(\sigma).$$

Let $\tau\in \tr(A)$ be such that
$$\tau(h)>\sigma,\quad h\in\mathcal H.$$
Without loss of generality, one may assume that 
$$\tau((f, g))=\alpha\tau_B(f)+\beta\int_{D^n\setminus S^{n-1}} g\mathrm{d}\mu,$$
where $\mu$ is a discrete probability measure of $D^n\setminus S^{n-1}$.
Since $\tau(\tilde{g}_\gamma)>\sigma,$ one has
$$\beta>\sigma.$$

Write $\mu=\sum_i\beta_i\delta_{x_i}$. For each $x_i$ with $\norm{x_i}\geq \frac{\gamma}{2}$, define $\tau_{B, x_i}\in\tr(B)$ by
$$\tau_{B, x_i}(f)=\frac{1}{k}\mathrm{Tr}(\varphi(f)(\frac{x_i}{\norm{x_i}})).$$
Then define the trace 
\begin{eqnarray*}
\tilde{\tau}((f, g)) & = & \alpha\tau_B(f)+ \beta(\sum_{\norm{x_i}\geq \frac{\gamma}{2}}\beta_i\tau_{B, x_i}(f)) + \beta(\sum_{\norm{x_i}<\frac{\gamma}{2}}\beta_i)\int_{(D^n)^\circ} g\mathrm{d}(\frac{1}{\sum_{\norm{x_i}<\frac{\gamma}{2}}\beta_i}\sum_{\norm{x_i}<\frac{\gamma}{2}}\beta_i\delta_{x_i})\\
& = & \alpha'\tilde{\tau}_B(f) + \beta' \int_{(D^n)^\circ} g\mathrm{d}\mu',
\end{eqnarray*}
where $$\tilde{\tau}_B(f)=\frac{\alpha\tau_B(f)+ \beta(\sum_{\norm{x_i}\geq \frac{\gamma}{2}}\beta_i \tau_{B, x_i}(f))}{\alpha+ \sum_{\norm{x_i}\geq \frac{\gamma}{2}}\beta\beta_i},\quad \alpha'=\alpha+ \sum_{\norm{x_i}\geq \frac{\gamma}{2}}\beta\beta_i,$$
and
$$\mu'= \frac{1}{\sum_{\norm{x_i}<\frac{\gamma}{2}}\beta_i}\sum_{\norm{x_i}<\frac{\gamma}{2}}\beta_i\delta_{x_i}, \quad \beta'= \beta(\sum_{\norm{x_i}<\frac{\gamma}{2}}\beta_i).$$

Note that
\begin{equation}\label{k-fix-0}
\tau(p)=\tilde{\tau}(p),\quad p\in\Kzero(A).
\end{equation}

By the choice of $\gamma$, one has
\begin{equation}\label{pert-mov}
\abs{\tilde{\tau}(f)-\tau(f)}<\eps/2,\quad f\in\mathcal F.
\end{equation}
Noting that
$$\tilde{\tau}(\tilde{g}_{2\gamma})=\tau(\tilde{g}_{2\gamma})\quad\textrm{and}\quad \tilde{\tau}(1-\tilde{g}_{\gamma/2})\geq\tau(1-\tilde{g}_{\gamma/2}),$$
one has
$$\tilde{\tau}(\tilde{g}_{2\gamma})>\sigma\quad\textrm{and}\quad \tilde{\tau}(1-\tilde{g}_{\gamma/2})>\sigma.$$
A straightforward calculation shows that
$$\beta'\geq\tilde{\tau}(\tilde{g}_{2\gamma})\quad \textrm{and}\quad \alpha'=\tilde{\tau}(1-\tilde{g}_{\gamma/2}),$$
and therefore
\begin{equation}\label{lbd-coef}
\beta'>\sigma\quad\textrm{and}\quad \alpha'>\sigma.
\end{equation}

Also noting that for any $h\in\mathcal H_B$, one has
$$\tilde{\tau}(\tilde{h})\geq\tau(\tilde{h})>\sigma,$$
and
$$\tilde{\tau}(\tilde{h})=\alpha'\tilde{\tau}_B(h).$$
Therefore,
$$\tilde{\tau}_B(h)>\sigma/\alpha'>\sigma,\quad h\in\mathcal H_B.$$
By the inductive assumption, there is $\tau'_B\in\tr(B)$ 
such that
\begin{equation}\label{k-fix}
\tau_B'(p)=\tau_B(p),\quad p\in\Kzero(B),
\end{equation}
\begin{equation}\label{pert-ind}
\abs{\tau'_B(f_B)-\tilde{\tau}_B(f_B)}<\eps/2,\quad f\in\mathcal F,
\end{equation}
and 
\begin{equation}\label{lbd-ind}
\alpha_{i}(\tau'_B)>\delta_B(\sigma),\quad  \textrm{all $i$}.
\end{equation}
Then the trace
$$\tau'(f, g)=\alpha'\tau'_B(f) + \beta' \int_{(D^n)^\circ} g\mathrm{d}\mu'$$
satisfies the condition of the lemma. By \eqref{k-fix-0} and \eqref{k-fix}, one has
$$\tau'(p)=\tau(p),\quad p\in\Kzero(A),$$
Indeed, by \eqref{pert-mov} and \eqref{pert-ind}, one has
\begin{equation}
\abs{\tau'(f)-\tau(f)}<\eps,\quad f\in\mathcal F.
\end{equation}
By \eqref{lbd-coef} and \eqref{lbd-ind}, one has
$$\alpha_{i}(\tau')>\sigma\delta_B(\sigma),\quad   \textrm{all $i$},$$
as desired.
\end{proof}

\begin{lem}\label{mat-pert}
Let $P$ be a $m\times n$ matrix. Assume that $m\geq n$ and $\mathrm{rank}(P)=n$. Let $\sigma>0$ be given. Then, for any $\eps>0$,  there is $\delta>0$ such that for any vectors $\xi\in\Real^m$ and $\kappa\in\Real^n$ with
\begin{itemize}
\item[(1)] $\xi_i>\sigma$, $i=1, ..., m$, and
\item[(2)] $\norm{\kappa-P\xi}_\infty<\delta,$
\end{itemize}
there is $\zeta\in\Real^m$ with $\zeta_i\geq 0$, $i=1, ..., m$ such that
\begin{itemize}
\item[(3)] $P\zeta=\kappa$, and
\item[(4)] $\|\xi-\zeta\|_\infty<\eps$.
\end{itemize}
\end{lem}

\begin{proof}
Regard $P$ as a linear map from $\Real^m\to \Real^n$. Since $\mathrm{rank}(P)=n$, one has that $P$ is surjective. Therefore,
for the given $\epsilon<\sigma$, there is $\delta>0$ such that
\begin{equation}\label{gap}
\mathrm{B}_{\norm{\cdot}_\infty}(0, \delta)\subseteq P(\mathrm{B}_{\norm{\cdot}_\infty}(0, \eps)).
\end{equation}
Then the $\delta$ satisfies the condition of the lemma. Indeed, let $\xi\in\Real^m$ and $\kappa\in\Real^n$ be given, and satisfy 
$$\norm{\kappa-P\xi}_\infty<\delta.$$ 
Then 
$$\kappa-P(\xi)\in\mathrm{B}_{\norm{\cdot}_\infty}(0, \delta)\subseteq \Real^n.$$
By \eqref{gap}, there is $\theta\in\mathrm{B}_{\norm{\cdot}_\infty}(0, \eps)\subseteq \Real^m$ such that
$$P(\theta)=\kappa-P(\xi),$$
and hence
$$P(\xi+\theta)=\kappa.$$
Since each entries of $\xi$ is at least $\sigma>\eps,$ one has that $$\xi+\theta\in(\Real^+)^m,$$
and therefore $$\zeta=\xi+\theta$$
is the desired vector.
\end{proof}

\begin{rem}
Let $A$ be an NCCC, and let $\kappa: \Kzero(A)\to \Ratn$ be a positive map. Since $A$ is of type I, it is exact; therefore, $\kappa$ (regarded as a state of $\Kzero(A)$) is induced by a trace. That is, there is $\tau\in \tr(A)$ such that 
$$\kappa(p)=\tau(p),\quad p\in\Kzero(A).$$
In particular, this implies that $\kappa$ factors through $\rho_A(\Kzero(A))$.
\end{rem}

\begin{lem}\label{pert-tr}
Let $A$ be an NCCC. Let $(\mathcal F, \eps)$ be given.  Let $p_1, p_2, ..., p_n\in\Kzero(A)$ be such that $\{p_1, ..., p_n\}$ is a set of generators for $\rho_A(\Kzero(A))$ (as an abelian group) (still use the same notation for the image of $p_i$). Then, there is a finite set $\mathcal H\subseteq A^+$ such that for any $\sigma>0$, there is $\delta>0$ such that if $\kappa: \Kzero(A)\to\Ratn$ and $\tau\in \tr(A)$ satisfy
\begin{itemize}
\item[(1)] $\tau(h)>\sigma$, $h\in \mathcal H,$ and
\item[(2)] $\abs{\tau(p_i)-\kappa(p_i)}<\delta$, $i=1, ..., n,$
\end{itemize}
then there is $\tau'\in \tr(A)$ such that
\begin{itemize}
\item[(3)] $\abs{\tau'(f)-\tau(f)}<\eps$, $f\in\mathcal F,$ and
\item[(4)] $ \tau'(p_i)=\kappa(p_i)$, $i=1, ..., n.$
\end{itemize}
\end{lem}
\begin{proof}
Without loss of generality, one may assume that $\mathcal F$ is contained in the unit ball of $A$. One may also assume that $p_1, p_1, ..., p_n\in \rho_A(\Kzero(A))$ are $\Ratn$-independent.

Let $\mathcal H\subseteq A^+$ be the subset of Lemma \ref{lbd-test} with respect to $A$, $\mathcal F$, and $\eps/2$. Then $\mathcal H$ satisfies the lemma.

In fact, for any $\sigma>0$, let $\delta'$ be the constant of  Lemma \ref{lbd-test} with respect to $\sigma$.
Consider the $m\times n$ matrix
$$
P=\left(
\begin{array}{cccc}
\frac{\mathrm{rank}(\Pi_1(p_1))}{d_1} & \frac{\mathrm{rank}(\Pi_2(p_1))}{d_2} & \cdots & \frac{\mathrm{rank}(\Pi_m(p_1))}{d_m} \\
\frac{\mathrm{rank}(\Pi_1(p_2))}{d_1} & \frac{\mathrm{rank}(\Pi_2(p_2))}{d_2} & \cdots & \frac{\mathrm{rank}(\Pi_m(p_2))}{d_m} \\
\vdots & \vdots & \vdots  & \vdots \\
\frac{\mathrm{rank}(\Pi_1(p_n))}{d_1} & \frac{\mathrm{rank}(\Pi_2(p_n))}{d_2} & \cdots & \frac{\mathrm{rank}(\Pi_m(p_n))}{d_m}
\end{array}
\right).
$$
Since $p_1, p_1, ..., p_n\in \rho_A(\Kzero(A))$ are $\Ratn$-independent, the matrix $P$ has rank $n$ and it satisfies the assumptions of Lemma \ref{mat-pert}. Applying Lemma \ref{mat-pert} to $\delta'$ (in the place of $\sigma$) and $\eps/2m$, one obtains $\delta$. 

Let $(\kappa, \tau)$ be a pair satisfies 
$$\abs{\tau(p_i)-\kappa(p_i)}<\delta,\quad i=1, ..., n,$$
and
\begin{equation}\label{lbd-cond}
\tau(h)>\sigma,\quad h\in \mathcal H.
\end{equation}
By \eqref{lbd-cond} and Lemma \ref{lbd-test}, there is $\tau'\in\tr(A)$ such that 
\begin{equation}\label{keep-proj}
\tau'(p_i)=\tau(p_i),\quad i=1, ..., n,
\end{equation}
$$\abs{\tau'(f)-\tau(f)}<\eps/2,\quad f\in\mathcal F,$$
and
$$\alpha_{i}(\tau')>\delta',\quad\forall i.$$

Set $$\xi=(\alpha_1(\tau'), \alpha_2(\tau'), ..., \alpha_m(\tau'))^{\mathrm{T}}\quad\mathrm{and}\quad \kappa=(\kappa(p_1), \kappa(p_2), ..., \kappa(p_n))^{\mathrm{T}},$$ and then one has
\begin{eqnarray*}
\norm{P\xi-\kappa}_\infty & = & \max\{\tau'(p_i) - \kappa(p_i): i=1, ..., n\} \\
& = & \max\{\tau(p_i) - \kappa(p_i): i=1, ..., n\} \quad \textrm{(by \eqref{keep-proj})} \\
& < & \delta
\end{eqnarray*}
By Lemma \ref{mat-pert}, there is a positive vector $\zeta=(\beta_1, ... , \beta_m)$ such that 
\begin{equation}\label{match-K}
P\zeta=(\kappa(p_1), \kappa(p_2), ..., \kappa(p_n))^{\mathrm{T}}
\end{equation}
and
$$\abs{\alpha_{i}-\beta_{i}}<\eps/2m,\quad \forall i.$$
Consider the trace
$$\tau''=\sum\beta_{i}\mu_{i}.$$
It is clear that 
$$\abs{\tau''(f)-\tau(f)}<\eps,\quad f\in\mathcal F.$$
By \eqref{match-K}, one has
$$\tau''(p)=\kappa(p),\quad p\in\Kzero(A).$$
Moreover, the trace $\tau''$ is indeed a state since $\tau''(1_A)=\kappa([1_A])=1$, as desired.
\end{proof}

\begin{thm}\label{est-AM}
Let $A$ be a unital NCCC, and let $\sigma>0$. Let $(\mathcal F, \eps)$ be given. Let $p_1, p_2, ..., p_n\in\Kzero(A)$ be such that the set $\{p_1, ..., p_n\}$ generates the group $\rho_A(\Kzero(A))$ (let us still use the same notation for the image of $p_i$). Then, there is a finite set $\mathcal H\subseteq A^+$ such that for any $\sigma>0$, there is $\delta>0$ such that if $\kappa: \Kzero(A)\to\Ratn$ and $\tau\in \tr(A)$ satisfy
\begin{itemize}
\item[(1)] $\abs{\tau(p_i)-\kappa[p_i]}<\delta$, $i=1, ..., n$, and 
\item[(2)] $\tau(h)>\sigma$, $h\in \mathcal H,$
\end{itemize} 
then there is a homomorphism $\phi: A\to Q$ such that
\begin{itemize}
\item[(3)] $[\phi]_0=\kappa$, and
\item[(4)] $\abs{\tau(f)-\mathrm{tr}(\phi(f))}<\eps$, $f\in\mathcal F,$
\end{itemize}
where $\mathrm{tr}$ is the canonical trace of $Q$. Moreover, $\phi$ can be chosen to have finite dimensional range.
\end{thm}

\begin{proof}
Let $\mathcal H\subseteq A^+$ denote the finite set of Lemma \ref{pert-tr} with respect to the data $A$, $\mathcal F$, $\eps/2$. Then $\mathcal H$ satisfies the theorem.

Indeed, given $\sigma>0$, consider the constant $\delta$ of Lemma \ref{pert-tr} with respect to $\sigma$.
Let $(\kappa, \tau)$ be a pair as above---$\delta$-compatible on $p_i$, $i=1, ..., n$, and such that $$\tau(h)>\sigma,\quad h\in\mathcal H.$$ By Lemma \ref{pert-tr}, there is $\tau'\in \tr(A)$ such that the pair $(\kappa, \tau')$ is exactly  compatible on the $p_i$ and
$$\abs{\tau'(f)-\tau(f)}<\eps/2,\quad f\in\mathcal F.$$ Then, by Lemma \ref{est-exact}, there is a homomorphism $\phi: A\to Q$ such that
$$[\phi]_0=\kappa$$
and
$$\abs{\mathrm{tr}(\phi(f))-\tau'(f)}<\eps/2,\quad f\in\mathcal F,$$
and therefore satisfying the condition of the theorem.
\end{proof}

\section{A decomposition theorem and a uniqueness theorem}

Recall (see, for instance, \cite{HV-Marriage-Lemma})
\begin{lem}[The Marriage Lemma]\label{PL}
Let $A$ and $B$ be two finite subsets of a metric space.  Suppose that for any set $X\subseteq A$, one has 
$$\#\{y\in B: \mathrm{dist}(y,X)<\eps\}\geq \#X,$$ 
then there is a set $B'\subseteq B$ and a one-to-one pairing between $A$ and $B'$ such that the distance between  the points in each pair is at most $\eps$.  
\end{lem}

The following statement is a generalization of the Marriage Lemma due to Gong in a unpublished manuscript:
\begin{lem}\label{PLG}
Let $A, B$ be two finite subsets of a metric space with subsets $A_1\subset A$ and $B_1\subset B$.  Suppose that for each $X\subset A_1$, one has
$$\# \{y\in B:  \mathrm{dist}(y,X)<\eps\}\geq \#X $$
and for each $Y\subset B_1$, one has
$$\# \{x\in A: \mathrm{dist}(x,Y)<\eps\}\geq \#Y.$$
Then there are sets $A_2\subseteq A$ and $B_2\subseteq B$ such that
$$A_1\subseteq A_2\quad\textrm{and}\quad B_1\subseteq B_2,$$
and the elements of $A_2$ and $B_2$ can be (bijectively) paired to within $\eps$.
\end{lem}

\begin{proof}

The proof is similar to that of \cite{HV-Marriage-Lemma}.  We spell out the details for reader's convenience.  In the case $\#(B_1)=0$, i.e., $B_1$ is empty, then this is the case of the classical Marriage Lemma (Lemma \ref{PL}) in the case of $A_1$ and $B$.  (For this case, $X_0$ can be chosen to be $A_1$.)

We define partial order for $(m,n)\in \Int^+\times\Int^+$ (where $\Int^+=\{0,1,2,...\}$ is the set of positive integers) by $(m,n)\leq (m_1,n_1)$ if $m\leq m_1$ and $n\leq n_1$.  We denote $(m,n)< (m_1,n_1)$ if $(m,n)\leq (m_1,n_1)$ and $(m,n)\not= (m_1,n_1)$.  We will prove the lemma by induction on $\#(A_1)$ and $\#(B_1)$.  That is, we assume that if the result is true for the case $\big(\#(A_1),\#(B_1)\big) < (m,n)$ and prove the lemma to be true for the case
$$\#(A_1)=m\quad \mbox{and} \quad \#(B_1)=n.$$

The rest of the proof divides into two cases.

\underline{Case 1}.  For any nonempty set $X\subseteq A_1$,
$$\#\{y\in B;~ \mathrm{dist}(y,X)<\eps\}\geq \#(X)+1$$
and for  any nonempty set $Y\subseteq B_1$,
$$\#\{x\in A;~ \mathrm{dist}(x,Y)<\eps\}\geq \#(Y)+1.$$
Choose any $a\in A_1$, there is a $b\in B$ such that
$\mathrm{dist}(a,b)<\eps$.  We will pair $a\in A_1$ with $b\in B$.  Then let $\tilde A=A\setminus \{a\}$ with $\tilde A_1=A_1\setminus \{a\}$, and
$\tilde B=B\setminus \{b\}$ with $\tilde B_1=B_1$ if $b\not\in B_1$ or $\tilde B_1=B_1\setminus \{b\}$ if $b\in B_1$.  It is easy to verify that $\tilde A\supseteq \tilde {A_1}$ and $\tilde B\supseteq \tilde {B_1}$
satisfy the condition of the lemma with $\big(\#(\tilde A_1),\#(\tilde B_1)\big)=$
 either $(m-1,n)$ or $(m-1,n-1)$.  That is, $\big(\#(A_1),\#(B_1)\big) < (m,n)$.  By the induction assumption, there is a subset $\tilde A_2\supseteq \tilde A_1$ and $\tilde B_2\supseteq \tilde B_1$ such that $\tilde X_0$ can be paired one by one within $\eps$.  Then the sets $A_2=\tilde A_2\cup \{a\}$ and $B_2=\tilde B_2\cup \{b\}$ satisfy the lemma.

 \underline{Case 2}. The conditions of Case 1 do not hold.  Then either there is $X\subseteq A_1$ or $Y\subseteq B_1$ does not satisfy the condition in Case 1.  That is, either there is $X\subseteq A_1$ such that
 $$\#\{y\in B;~ \mathrm{dist}(y,X)<\eps\}= \#(X),$$
 or there is  $Y\subseteq B_1$ such that
 $$\#\{x\in A;~ \mathrm{dist}(x,Y)<\eps\}= \#(Y).$$
 
Without loss of generality, let us assume there is $Y_1\subseteq B_1$ such that
 $$\#\{ x\in A;~ \mathrm{dist}(x, Y_1)<\eps\} =\# (Y_1).$$
 Let $X_1=\{x\in A;~ \mathrm{dist}(x, Y_1)<\eps\}$.  Then for any subset $Z\subseteq Y_1$, $\{ x\in A;~ \mathrm{dist}(x, Z)<\eps\}\subseteq X_1$ and therefore $\{ x\in A;~ \mathrm{dist}(x, Z)<\eps\}=\{ x\in X_1;~ \mathrm{dist}(x, Z)<\eps\}$ which has at least $\#(Z)$ elements.  That is, $Y_1$ and $X_1$ satisfy  the condition for the classical Pairing Lemma with $Y_1=A,~ X_1=B$ as in 2.10.  Hence $Y_1$ and $X_1$ can be paired one by one to within $\eps$.

 Let $\tilde B=B\setminus Y_1$ with subset $\tilde B_1=B_1\setminus Y_1$, and let
$\tilde A=A\setminus X_1$ with subset $\tilde A_1=A_1\setminus (X_1\cap A_1)$.
Let us verify that $\tilde B\supseteq \tilde B_1$ and $\tilde A\supseteq \tilde A_1$ satisfy the condition of the lemma.  Let $Y\subseteq \tilde B_1$.  Then
$$\{ x\in  A;~ \mathrm{dist}(x, Y\cup Y_1)<\eps\}= \{ x\in \tilde A_1;~ \mathrm{dist}(x, Y)<\eps\}\cup X_1.$$
Since $\#(X_1)=\#(Y_1)$ and $\{ x\in  A;~ \mathrm{dist}(x, Y\cup Y_1)<\eps\}\geq \#(Y)+\#(Y_1)=
\#(Y)+\#(X_1)$, we have,  $\{ x\in \tilde A_1;~ \mathrm{dist}(x, Y)<\eps\}\geq  \#(Y)$.

Let $X\subset \tilde A_1$.  Since for point $y\in Y_1,~ \mathrm{dist}(y, X)\geq \eps$, we have
$$\# \{ y\in \tilde B;~ \mathrm{dist}(y,X)<\eps\}=\# \{ y\in  B;~ \mathrm{dist}(y,X)<\eps\} \geq \#(X).$$
So the conditions of our lemma hold for $\tilde A\supseteq \tilde A_1$ and $\tilde B\supseteq \tilde B_1$ with
$$\#(\tilde B_1)=\#(B_1)-\#(Y_1)<n
\quad \mbox{and }\quad
\#(\tilde A_1)\leq\#(A_1)\leq m.
$$
So by the inductive assumption, there exist
$\tilde A_2\supseteq \tilde A_1$ and $\tilde B_2\supseteq \tilde B_1$ such that
$\tilde A_2$ and $\tilde B_2$ can be paired element by element to within $\eps$.  Then the sets
$A_2=\tilde A_2\cup X_1$ and $B_2=\tilde B_2\cup Y_1$ satisfy the lemma.
\end{proof}

\begin{defn}[See 2.2 of \cite{GLN-TAS}]
Let $A$ be a unital C*-algebra with $\tr(A)\neq\O$. Recall that each self-adjoint $a\in A$ induces $\hat{a}\in \aff(\tr(A))$ by $\hat{a}(\tau)=\tau(a)$, $\tau\in\tr(A)$. Denote this map by $\rho_A$.

Denote by $A^+_{1}$ the set of positive elements with norm at most $1$, and then denote by $A^+_{1, q}$ the image of $A^+_{1}$ in $\aff(\tr(A))$ under the canonical map $\rho_A$.
\end{defn}

As a consequence of this generalized version of the Marriage Lemma, one has

\begin{lem}\label{pert-disk}
Let $A=B\oplus_{\MC{k}{S^{n-1}}} \MC{k}{D^n}$ be an NCCC with $n\geq 1$. Let $\Delta: A^+_{1, q}\setminus\{0\}\to (0, 1)$ be an order preserving map. Let $\mathcal F\subseteq A$ be a finite set, and let $\eps>0$, $1>\gamma>0$, and $M\in\mathbb N$. There are finite sets $\mathcal H_1, \mathcal H_2 \subseteq A^+$ and $\delta>0$ such that for any unital homomorphisms $\phi, \psi: A\to \Mat{m}$ such that
\begin{itemize}
\item[(1)] $\tau(\phi(h)), \tau(\psi(h))>\Delta(\hat{h})$, $h\in \mathcal H_1$, and
\item[(2)] $\abs{\tau(\phi(h))-\tau(\psi(h))}<\delta$, $h\in\mathcal H_2$,
\end{itemize}
there are unital homomorphisms $\phi', \psi': A\to \Mat{m}$ such that
\begin{itemize}
\item[(3)] $\norm{\phi'(f)-\phi(f)}<\eps$, $\norm{\psi'(f)-\psi(f)}<\eps$, $f\in\mathcal F,$
\item[(4)] $\mathrm{SP}(\phi')\cap \mathrm{B}(1-\gamma)=\mathrm{SP}(\psi')\cap \mathrm{B}(1-\gamma)$, and each point in $\mathrm{SP}(\phi')\cap \mathrm{B}(1-\gamma)$ has multiplicity at least $M$, 
\end{itemize}
where $B(1-\gamma)\subseteq D^n$ is the closed ball with radius $1-\gamma$.
\end{lem}


\begin{proof}
Let $(\mathcal F, \eps)$ be given. For each $f\in\mathcal F$, write $f=(f_B, f_D)$ where $f_B\in B$ and $f_D\in \MC{k}{D^n}$.
Set $$\mathcal F_B=\{f_B: f\in\mathcal F\}\quad\textrm{and}\quad \mathcal F_D=\{f_D: f\in\mathcal F\}.$$ 
Choose $\eta>0$ such that for any $x, y\in D^n$ with $\mathrm{dist}(x, y)<4\eta$, one has
\begin{equation}\label{choose-eta}
\norm{f_D(x)-f_D(y)}<\eps,\quad f\in\mathcal F.
\end{equation}
Without loss of generality, one may assume that $5\eta<\gamma$.

Fix an $\eta$-dense finite subset $\{x_1, x_2, ..., x_l\}\subseteq \mathrm{B}(1-\gamma)$. For each subset $$\{y_1, y_2, ..., y_t\}\subseteq\{x_1, x_2, ..., x_l\},$$ 
define
$$ h_{\{y_1, ..., y_t\}}=\max\{1-\frac{1}{\eta}\mathrm{dist}(x, \{y_1, ..., y_t\}_\eta), 0\},$$
where $Y_\eta$ denotes the $\eta$-neighborhood of $Y$ if $Y$ is a subset of a metric space.
Using
$$\{y_1, y_2, ..., y_t\}_{2\eta} \neq \{y_1, y_2, ..., y_t\}_{3\eta},$$
choose a positive function $g_{\{y_1, ..., y_t\}}\in \mathrm{C}(D^n)$ such that $0<g_{\{y_1, ..., y_t\}}\leq 1$ and
$$\mathrm{supp}(g_{\{y_1, ..., y_t\}})\subseteq \{y_1, ..., y_t\}_{3\eta} \setminus \{y_1, ..., y_t\}_{2\eta}.$$
For functions $h_{y_1, ..., y_t}$ and $g_{y_1, ..., y_t}$, regard them as the elements 
$$(0, h_{y_1, ..., y_t}1_{\MC{k}{D^n}}),\  (0, g_{y_1, ..., y_t}1_{\MC{k}{D^n}}) \in A=B\oplus_{\MC{k}{S^{n-1}}}\MC{k}{D^n},$$
and still denote them by $h_{y_1, ..., y_t}$ and $g_{y_1, ..., y_t}$, respectively.

Put 
$$\tilde{\mathcal H}_1=\{g_{\{y_1, ..., y_t\}}: \{y_1, ..., y_t\}\subseteq\{x_1, x_2, ..., x_l\}\},$$
$$\mathcal H_2=\{h_{\{y_1, ..., y_t\}}: \{y_1, ..., y_t\}\subseteq\{x_1, x_2, ..., x_l\}\},$$
and
\begin{equation}\label{defn-delta}
\delta=\min\{\Delta(\hat{g}_{\{y_1, ..., y_t\}}): \{y_1, ..., y_t\}\subseteq\{x_1, ..., x_l\}\}\}.
\end{equation}
Also pick a finite open cover $\mathcal U$ of $B(1-\gamma)$ such that each $U\in\mathcal U$ has diameter at most $\eta$,
$$\bigcup_{U\in\mathcal U} U\subseteq B(1-\gamma/2)\quad\textrm{and}\quad U\setminus \bigcup_{V\neq U} V \neq \O.$$ Then choose continuous functions $s_{U, 1}, ..., s_{U, M}\in\mathrm{C}(D^n)$ such that 
$$\mathrm{supp}(s_{U, i})\subseteq \bigcup_{V\neq U} V\quad\textrm{and}\quad \mathrm{supp}(s_{U, i})\cap \mathrm{supp}(s_{U, j})=\O,\ i\neq j.$$
Regard each $s_{U, i}$ as an element of $A$, and put
$$\mathcal S= \{s_{U, i}: U\in\mathcal U, i=1, ..., M\}.$$

Then $\mathcal H_1:=\tilde{\mathcal H}_1\cup\mathcal S$, $\mathcal H_2$ and $\delta$ have the property stated in the conclusion of the lemma. 

Indeed, let $\phi, \psi: A\to \Mat{m}$ be unital homomorphisms
satisfying
\begin{equation}\label{density-cond-ML}
 \tau(\phi(h)), \tau(\psi(h))>\Delta(\hat{h}),\quad h\in \mathcal H_1,
\end{equation} 
and
\begin{equation}\label{tr-cond-ML}
 \abs{\tau(\phi(h))-\tau(\psi(h))}<\delta,\quad h\in \mathcal H_2.
\end{equation}

Let $Y\subseteq \mathrm{SP}(\phi)\cap \mathrm{B}(1-\gamma)$ be an arbitrary subset. Let 
$$\{y_1, y_2, ..., y_t\}\subseteq \{x_1, x_2, ..., x_l\}$$
denote the subset of the points $y$ satisfying $\mathrm{dist}(y, Y)<\eta$. Then 
\begin{eqnarray*}
\#Y & \leq & m \cdot \mathrm{tr}(\phi(h_{\{y_1, ..., y_t\}})) \\
      & \leq & m \cdot \mathrm{tr}(\psi(h_{\{y_1, ..., y_t\}})) + m \cdot \delta \quad\quad \textrm{(by \eqref{tr-cond-ML})}\\
      & \leq & m \cdot \mathrm{tr}(\psi(h_{\{y_1, ..., y_t\}})) + m \cdot \Delta(\hat{g}_{y_1, ..., y_t})\quad\quad \textrm{(by \eqref{defn-delta})}\\
      & \leq & \#(\mathrm{SP}(\psi)\cap\{y_1, ..., y_t\}_{2\eta}) + m \cdot \Delta(\hat{g}_{y_1, ..., y_t}) \\
      & \leq & \#(\mathrm{SP}(\psi)\cap\{y_1, ..., y_t\}_{3\eta}) \quad\quad \textrm{(by \eqref{density-cond-ML})}\\
      & \leq & \#(\mathrm{SP}(\psi)\cap Y_{4\eta}).
\end{eqnarray*}

The same argument shows that 
$$\#X\leq \#(\mathrm{SP}(\phi)\cap X_{4\eta}),\quad X\subseteq \mathrm{SP}(\psi)\cap\mathrm{B}(1-\gamma).$$

Then, applying Lemma \ref{PLG} with 
$$A=\mathrm{SP}(\phi)\cap B(1-\eta),\quad A_1=\mathrm{SP}(\phi)\cap B(1-\gamma)$$ and 
$$B=\mathrm{SP}(\psi)\cap B(1-\eta),\quad B_1=\mathrm{SP}(\psi)\cap B(1-\gamma),$$
one obtains $A_2$ and $B_2$ such that 
\begin{equation}\label{pairing-A}
\mathrm{SP}(\phi)\cap B(1-\gamma)\subseteq A_2\subset \mathrm{SP}(\phi)\cap B(1-\eta) 
\end{equation}
and
\begin{equation}\label{pairing-B}
\mathrm{SP}(\psi)\cap B(1-\gamma)\subseteq B_2\subset \mathrm{SP}(\psi)\cap B(1-\eta),
\end{equation}
and $A_2$ and $B_2$ can be paired up to $4\eta$.

Write $$A_2=\{z_{\phi, 1}, z_{\phi, 2}, ..., z_{\phi, s}\}\quad\mathrm{and}\quad B_2=\{z_{\psi, 1}, z_{\psi, 2}, ..., z_{\psi, s}\}$$
for some $s$, where 
\begin{equation}\label{small-dist}
\mathrm{dist}(z_{\phi, i}, z_{\psi, i})<4\eta,\quad i=1, ..., s.
\end{equation}

Then, up to unitary equivalence, there are decompositions
$$\phi=\tilde\phi\oplus\bigoplus_{i=1}^s\pi_{z_{\phi, i}}\quad\mathrm{and}\quad \psi=\tilde\psi\oplus\bigoplus_{i=1}^s\pi_{z_{\psi, i}},$$
where $\mathrm{SP}\tilde{\phi}\cap\mathrm{B}(1-\gamma)=\mathrm{SP}\tilde{\psi}\cap\mathrm{B}(1-\gamma)=\O$, and the homomorphisms $$\phi'=\phi=\tilde\phi\oplus\bigoplus_{i=1}^s\pi_{z_{\phi, i}}\quad\mathrm{and}\quad \psi'=\tilde\psi\oplus\bigoplus_{i=1}^s\pi_{z_{\phi, i}}$$
have the required properties except the requirement on multiplicity.
 
Indeed, by \eqref{small-dist} and \eqref{choose-eta}, one has
$$\norm{\phi(f)-\phi'(f)}=0<\eps\quad\mathrm{and}\quad\norm{\psi(f)-\psi'(f)}<\eps,\quad f\in\mathcal F.$$
By \eqref{pairing-A} and \eqref{pairing-B}, one has that
$$\mathrm{SP}(\phi')\cap\mathrm{B}(1-\gamma)=\{z_{\phi, 1}, ..., z_{\phi, s}\}\cap\mathrm{B}(1-\gamma)=\mathrm{SP}(\psi')\cap\mathrm{B}(1-\gamma).$$

To satisfy the multiplicity condition, one needs to perturb $\phi'$ and $\psi'$ further.
Since $$\phi(s_{U, i})> \Delta(\hat{s}_{U, i})>0,\quad U\in\mathcal U, i=1, ..., M,$$ one has that for any $U\in\mathcal U$, at least $M$ of $\{z_{\phi_1}, ..., z_{\phi_s}\}$ are inside $U\setminus\bigcup_{V\neq U} V$ (counting multiplicity). Then there is a grouping of  grouping $\{z_{\phi_1}, ..., z_{\phi_s}\}$ such that each group is insider at most one open set $U$ and has elements at least $M$ if it is covered by an open set $U$. Therefore, after a perturbation, one may assume that each point in $\{z_{\phi_1}, ..., z_{\phi_s}\}\cap B(1-\gamma)$ has multiplicity at least $M$, and hence $\phi'$ and $\psi'$ satisfy the desired multiplicity condition.
\end{proof}

\begin{thm}\label{decp-thm}
Let $A=B\oplus_{\MC{k}{S^{n-1}}} \MC{k}{D^n}$ be an NCCC, and let $\Delta: A^+_{1, q}\setminus\{0\}\to (0, 1)$ be an order preserving map. Let $\mathcal F\subseteq A$ be a finite set. Let $\eps>0$, $\eta>0$ and $K\in \mathbb{N}\setminus\{0\}$. There are finite sets $\mathcal H_1, \mathcal H_2 \subseteq A^+$ and $\delta>0$ such that for any unital homomorphisms $\phi, \psi: A\to \Mat{m}$ such that
\begin{enumerate}
\item[(1)] $\tau(\phi(h)), \tau(\phi(h))>\Delta(\hat{h}),\quad h\in \mathcal H_1$, and
\item[(2)] $  \abs{\tau(\phi(h))-\tau(\psi(h))}<\delta,\quad h\in\mathcal H_2$,
\end{enumerate}
there exist unital homomorphisms $\tilde{\phi}, \tilde{\psi}: A\to \Mat{m}$ such that
\begin{enumerate}
\item[(3)] $\norm{\tilde{\phi}(f)-\phi(f)}<\eps$, $\norm{\tilde{\psi}(f)-\psi(f)}<\eps$, $f\in\mathcal F,$
\item[(4)]  $\tilde{\phi}$ and $\tilde{\psi}$ have decompositions
$$\tilde{\phi}=\tilde{\phi}_0\oplus\underbrace{\tilde{\phi}_1\oplus\cdots\oplus \tilde{\phi}_1}_{K} \quad\textrm{and}\quad \tilde{\psi}=\tilde{\psi}_0\oplus\underbrace{\tilde{\psi}_1\oplus\cdots\oplus\tilde{\psi}_1}_K$$
\end{enumerate}
such that $\tilde{\phi}_1$ and $\tilde{\psi}_1$ are unitarily equivalent, and
$$\mathrm{tr}(\tilde{\phi}_0(a))\leq \eta\cdot \mathrm{tr}(\tilde{\phi}(a)) \quad \textrm{and} \quad  \mathrm{tr}(\tilde{\psi}_0(a))\leq \eta\cdot\mathrm{tr}(\tilde{\psi}(a)),\quad a\in\mathcal F.$$
\end{thm}

\begin{proof}
The statement holds if $A$ is finite dimensional. Assume that the statement holds for $B$, and let us show that the statement holds for $A$.

Let $(\mathcal F, \eps)$ be given. For each $f\in\mathcal F$, write $f=(f_B, f_D)$ where $f_B\in B$ and $f_D\in \MC{k}{D^n}$.
Set $$\mathcal F_B=\{f_B: f\in\mathcal F\}\quad\textrm{and}\quad \mathcal F_D=\{f_D: f\in\mathcal F\}.$$

For each $r>0$, define
$$g_r: [0, 1]\ni x\mapsto \left\{
\begin{array}{ll}
1, & x\in[0, 1-r], \\
\textrm{linear}, & x\in[1-r, 1-r/2],\\
0, & x\in[1-r/2, 1];
\end{array}
\right.
$$
Also define $\tilde{g}_r=(0, g_r(\norm{x})1_{\MC{k}{D^n}})\in A$.

Pick $\gamma>0$ such that 
\begin{equation}\label{defn-gamma-decp}
\norm{f_D(x)-f_D(y)}<\eps/8,\quad f\in\mathcal F,\ \mathrm{dist}(x, y)<\gamma,
\end{equation}
and consider $g_\gamma$.

Consider the linear map
$$\chi_\gamma: \MC{k}{S^{n-1}} \ni f \mapsto (x\mapsto (1-g_\gamma(\norm{x})) \cdot f(\frac{x}{\norm{x}}))\in \MC{k}{D^n}$$
Clearly, $\chi_\gamma$ is a positive linear map, and then the map
$$\Gamma_\gamma: B \ni f \mapsto (f, \chi_\gamma(\varphi (f)))\in A$$
is positive and injective, where $\varphi: B\to\MC{k}{S^{n-1}}$ is the gluing map.

If $f\in B$ satisfies that $\tau(f)=0$, $\tau\in\tr(B)$, then $\tau(\chi(f))=0$, $\tau\in\tr(\MC{k}{D^n})$, and therefore $$\tau(\Gamma_\gamma(f))=0,\quad \tau\in \tr(A).$$
Hence the map
$$\Delta_B: B_q^1\ni  f\mapsto \Delta(\Gamma_\gamma(f))\in(0, 1)$$
is well defined and is order preserving.

Applying the inductive hypothesis to $B$ with $\Delta_B$, $\mathcal F_B$, $\eps/2$, $K$ and $\eta$, 
one obtains $\mathcal H_{B, 1}$, $\mathcal H_{B, 2}$, and $\delta_B$.

Define
$$\tilde{\mathcal H}_{B, 1}=\{\Gamma_\gamma(h): \ h\in\mathcal H_{B, 1}\}\quad
\textrm{and}\quad 
\tilde{\mathcal H}_{B, 2}=\{\Gamma_\gamma(h): \ h\in\mathcal H_{B, 2}\}.$$

Put $$\sigma_B=\min\{\Delta(\hat{h}): h\in\tilde{\mathcal H}_{B, 1}\}.$$

Let $\tilde\gamma $ be a positive number such that 
\begin{equation}\label{defn-Tgamma-decp}
\norm{\chi_\gamma(h)(x)-\chi_\gamma(h)(y)}<\min\{\frac{\Delta(\widehat{1-\tilde{g}_{\gamma/2}})\delta_B}{4}, \frac{\sigma_B}{4}, \frac{\eps}{8}\},\quad h\in\mathcal H_{B, 1}\cup\mathcal H_{B, 2}\cup\mathcal F_D
\end{equation}
for any $x, y\in D^n$ satisfying $\mathrm{dist}(x, y)<\tilde{\gamma}$.

Let $\mathcal H_{D, 1}\subseteq A$, $\mathcal H_{D, 2}\subseteq A$ and $\delta_D$ denote the finite sets and constant of Lemma \ref{pert-disk} with respect to $\mathcal F\cup\tilde{\mathcal H}_{B, 1}\cup\tilde{\mathcal H}_{B, 2}$ (in the place of $\mathcal F$), $\min\{\eps/8, \Delta(\widehat{1-\tilde{g}_{\gamma/2}})\delta_B/8, \sigma_B/8\}$ (in the place of $\eps$), $\tilde{\gamma}$, $M=\lfloor 2K/\eta\rfloor +1$, and $\Delta$.

Then $$\mathcal H_1=\tilde{\mathcal H}_{B, 1}\cup\mathcal H_{D, 1}\cup\{1-{\tilde{g}_{\gamma/2}}, \tilde{g}_{\gamma}\},\quad \mathcal H_2= \tilde{\mathcal H}_{B, 2}\cup\mathcal H_{D, 2},$$
and 
$$\delta=\min\{\Delta(\widehat{1-\tilde{g}_{\gamma/2}})\delta_B/4, \sigma_B/4, \delta_D\}$$
satisfy the statement.

In fact, let $\phi, \psi: A\to \Mat{m}$ be unital homomorphisms satisfying
\begin{equation}\label{density-cond-decp}
\tau(\phi(h)), \tau(\phi(h))>\Delta(\hat{h}),\quad h\in \mathcal H_1
\end{equation}
and
\begin{equation}\label{tr-cond-decp}
 \abs{\tau(\phi(h))-\tau(\psi(h))}<\delta,\quad h\in\mathcal H_2.
\end{equation}

Since $\delta\leq\delta_D$, by Lemma \ref{pert-disk}, there are homomorphisms
$$\phi', \psi': A\to \Mat{m}$$
such that 
\begin{equation}\label{first-pert-phi}
\norm{\phi'(f)-\phi(f)}<\min\{\eps/8, \Delta(\widehat{1-g_{\gamma/2}})\delta_B/8, \sigma_B/8\}, \quad f\in\mathcal F\cup\tilde{\mathcal H}_{B, 1}\cup\tilde{\mathcal H}_{B, 2},
\end{equation}
\begin{equation}\label{first-pert-psi}
\norm{\psi'(f)-\psi(f)}<\min\{\eps/8, \Delta(\widehat{1-g_{\gamma/2}})\delta_B/8, \sigma_B/8\},\quad f\in\mathcal F\cup\tilde{\mathcal H}_{B, 1}\cup\tilde{\mathcal H}_{B, 2} ,
\end{equation}
and
$$\mathrm{Sp}(\phi')\cap \mathrm{B}(1-\tilde{\gamma})=\mathrm{Sp}(\psi')\cap \mathrm{B}(1-\tilde{\gamma})=\{x_1, x_2, ..., x_l\},$$
for some $x_1, x_2, ..., x_l\in \mathrm{B}(1-\tilde{\gamma})$ with multiplicity at least $M$. Therefore, up to unitary equivalence, there are decompositions
\begin{equation}\label{first-decp-phi}
\phi'=\phi'_B\oplus(\bigoplus_{i=1}^{m_\phi}\pi_{x_{\phi, i}})\oplus(\bigoplus_{i=1}^l\pi_{x_i}),
\end{equation}
\begin{equation}\label{first-decp-psi}
\psi'=\psi'_B\oplus(\bigoplus_{i=1}^{m_\psi}\pi_{x_{\psi, i}})\oplus(\bigoplus_{i=1}^l\pi_{x_i}).
\end{equation}
for some $m_\phi, m_\psi\in \mathbb N$ and some $x_{\phi, 1}, ..., x_{\phi, m_\phi}, x_{\psi, 1}, ..., x_{\psi, m_\psi}\in D^n$ with $1>\norm{x_{\phi, i}}\geq \tilde{\gamma}$ and $1>\norm{x_{\psi, j}}\geq \tilde{\gamma}$.

By \eqref{density-cond-decp}, \eqref{first-pert-phi} and \eqref{first-pert-psi}, one has 
\begin{equation}\label{first-pert-lbd}
\mathrm{tr}(\phi'(h)), \mathrm{tr}(\psi'(h)) >\frac{7}{8}\Delta(\hat{h}), \quad h\in\mathcal H_{B, 1}.
\end{equation}

It also follows from \eqref{tr-cond-decp} \eqref{first-pert-phi} and \eqref{first-pert-psi} that for any $h\in \tilde{\mathcal H}_{B, 1}\cup\tilde{\mathcal H}_{B, 2}$, 
\begin{eqnarray}\label{first-pert-dist}
\abs{\mathrm{tr}(\phi'(h))-\mathrm{tr}(\psi'(h))} & < & \abs{\mathrm{tr}(\phi'(h))-\mathrm{tr}(\psi'(h))} + \Delta(\widehat{1-\tilde{g}_{\gamma/2}})\delta_B/4\\
& < & \delta + \Delta(\widehat{1-\tilde{g}_{\gamma/2}})\delta_B/4 \leq \Delta(\widehat{1-\tilde{g}_{\gamma/2}})\delta_B/2.\nonumber
\end{eqnarray}
Therefore, by the decompositions \eqref{first-decp-phi} and \eqref{first-decp-psi}, one has
\begin{equation}\label{first-pert-dist-tr}
\abs{\mathrm{tr}(\phi'_B(h)\oplus(\bigoplus_{i=1}^{m_\phi}h(x_{\phi, i}))) - \mathrm{tr}(\psi'_B(h)\oplus(\bigoplus_{i=1}^{m_\psi}h(x_{\psi, i}))) } < \Delta(\widehat{1-\tilde{g}_{\gamma/2}})\delta_B/2
\end{equation}

For each point $x_{\phi, i}$ (or $x_{\psi, j}$), replace the homomorphism $\pi_{x_{\phi, i}}$ (or $\pi_{x_{\psi, i}}$) by the homomorphism $\pi_{x'_{\phi, i}}$ (or $\pi_{x'_{\psi, i}}$), where $x'_{\phi, i}=x_{\phi, i}/\norm{x_{\phi, i}}\in S^{n-1}$ (or $x'_{\psi, i}=x_{\psi, i}/\norm{x_{\psi, i}}\in S^{n-1}$). By the choice of $\tilde{\gamma}$ (see \eqref{defn-Tgamma-decp}), one has
\begin{equation}\label{sec-pert-pts-phi}
\norm{\bigoplus_{i=1}^{m_\phi}\pi_{x'_{\phi, i}}(h)-\bigoplus_{i=1}^{m_\phi}\pi_{x_{\phi, i}}(h)} < \min\{\frac{\Delta(\widehat{1-\tilde{g}_{\gamma/2}})\delta_B}{4}, \frac{\sigma_B}{4}, \frac{\eps}{8}\},\quad h\in \tilde{\mathcal H}_{B_1, 1}\cup\tilde{\mathcal H}_{B_1, 2}\cup\mathcal F_D,
\end{equation}
and 
\begin{equation}\label{sec-pert-pts-psi}
\norm{\bigoplus_{i=1}^{m_\psi}\pi_{x'_{\psi, i}}(h)-\bigoplus_{i=1}^{m_\psi}\pi_{x_{\psi, i}}(h)} < \min\{\frac{\Delta(\widehat{1-\tilde{g}_{\gamma/2}})\delta_B}{4}, \frac{\sigma_B}{4}, \frac{\eps}{8}\},\quad h\in \tilde{\mathcal H}_{B_1, 1}\cup\tilde{\mathcal H}_{B_1, 2}\cup\mathcal F_D.
\end{equation}

Put
$$\phi_B''=\phi'_B\oplus(\bigoplus_{i=1}^{m_\phi}\pi_{x'_{\phi, i}})\quad\textrm{and}\quad\psi_B''=\psi'_B\oplus(\bigoplus_{i=1}^{m_\psi}\pi_{x'_{\psi, i}}).$$
Since $x'_{\phi, i}, x'_{\psi, j}\in S^{n-1}$, the homomorphisms $\phi_B''$ and $\psi_B''$ factor through $B$.
Define
\begin{equation}\label{sec-decp}
\phi''=\phi''_B\oplus(\bigoplus_{i=1}^l\pi_{x_i})\quad\textrm{and}\quad \psi''=\psi''_B\oplus(\bigoplus_{i=1}^l\pi_{x_i}).
\end{equation}

By \eqref{sec-pert-pts-phi}, \eqref{sec-pert-pts-psi}, one has
\begin{equation}\label{sec-pert-F}
\norm{\phi'(f)-\phi''(f)}<\eps/8\quad\textrm{and}\quad \norm{\psi''(f)-\psi'''(f)}<\eps/8,\quad f\in\mathcal F.
\end{equation}

By \eqref{sec-pert-pts-phi}, \eqref{sec-pert-pts-psi}, and \eqref{first-pert-lbd}, one has 
\begin{equation}\label{second-pert-lbd}
\mathrm{tr}(\phi''(h)), \mathrm{tr}(\psi''(h)) >\frac{3}{4}\Delta(\hat{h}), \quad h\in\mathcal H_{B, 1}.
\end{equation}

By \eqref{sec-pert-pts-phi}, \eqref{sec-pert-pts-psi} and \eqref{first-pert-dist-tr}, one has
\begin{equation}\label{second-pert-tr}
\abs{\mathrm{tr}(\phi''(h)) - \mathrm{tr}(\psi''(h))} < \frac{3}{4}\Delta(\widehat{1-\tilde{g}_{\gamma/2}})\delta_B,\quad h\in \tilde{\mathcal H}_{B, 2}.
\end{equation}

For each point $x_i$ with $\norm{x_i}>1-\gamma$, replace the homomorphism $\pi_{x_i}$ by $\pi_{x_i'}$, where $x_i'=x_i/\norm{x_i}\in S^{n-1}$. Clearly, the homomorphism $\pi_{x_i'}$ factors through $B$. Also note that for each such $x_i$, one has
\begin{equation}\label{third-pert-inc-tr}
\mathrm{tr}(\pi_{x'_i}(h))>\mathrm{tr}(\pi_{x_i}(h)),\quad h\in\tilde{\mathcal H}_{B, 1}.
\end{equation}
By the choice of $\gamma$ (see \eqref{defn-gamma-decp}), one also has
\begin{equation}\label{small-F-3}
\norm{\pi_{x_i'}(f_D) - \pi_{x_i}(f_D)}<\eps/8,\quad f\in\mathcal F.
\end{equation}

Set
$$\phi'''_B=\phi_B''\oplus(\bigoplus_{\norm{x_i}\geq 1-\gamma}\pi_{x'_i})\quad\textrm{and}\quad \psi'''_B=\psi_B''\oplus(\bigoplus_{\norm{x'_i}\geq 1-\gamma}\pi_{x'_i}),$$
and consider
\begin{equation}\label{third-decp}
\phi'''=\phi'''_B\oplus(\bigoplus_{\norm{x_i}<1-\gamma}\pi_{x_i})\quad\textrm{and}\quad \psi''=\psi'''_B\oplus(\bigoplus_{\norm{x_i}<1-\gamma}^l\pi_{x_i}).
\end{equation}

By \eqref{small-F-3}, one has
\begin{equation}\label{third-pert-F}
\norm{\phi''(f)-\phi'''(f)}<\eps/8\quad\textrm{and}\quad \norm{\psi''(f)-\psi'''(f)}<\eps/8,\quad f\in\mathcal F.
\end{equation}

By \eqref{third-pert-inc-tr} and \eqref{second-pert-lbd}, one has
\begin{equation}\label{third-pert-lbd}
\mathrm{tr}(\phi'''(h)), \mathrm{tr}(\psi'''(h)) >\frac{3}{4}\Delta(\hat{h}), \quad h\in\mathcal H_{B, 1}.
\end{equation}
By \eqref{second-pert-tr}, one has
\begin{equation}\label{third-pert-tr}
\abs{\mathrm{tr}(\phi'''(h)) - \mathrm{tr}(\psi'''(h))} < \frac{3}{4}\Delta(\widehat{1-\tilde{g}_{\gamma/2}})\delta_B,\quad h\in \tilde{\mathcal H}_{B, 2}.
\end{equation}

Denote by
$$N=\mathrm{rank}(\phi_B'''(1))=\mathrm{rank}(\psi_B'''(1)).$$
Then, by \eqref{third-pert-lbd},
$$\frac{N}{m}\geq \mathrm{tr}(\phi'''(1-\tilde{g}_{\gamma/2}))>\frac{3}{4}\Delta(\widehat{1-\tilde{g}_{\gamma/2}}),$$
and
$$\frac{m-N}{m} \geq \mathrm{tr}(\phi'''(g_\gamma))>\Delta(\widehat{(\tilde{g}_\gamma)}).$$
Therefore
\begin{equation}\label{rank-bd}
\frac{3}{4}\Delta(\widehat{1-\tilde{g}_{\gamma/2}})  < \frac{N}{m} < 1- \Delta(\widehat{(\tilde{g}_\gamma)}). 
\end{equation}

Then, by considering the unital homomorphisms $\phi_B''', \psi_B''': B \to\Mat{N}$, it follows from \eqref{third-pert-tr} and \eqref{rank-bd} that 
\begin{equation}
\frac{1}{N}(\mathrm{Tr}(\phi_B'''(h))-\mathrm{Tr}(\phi_B'''(h))) =\frac{m}{N}(\mathrm{tr}(\phi'''(\Gamma_\gamma(h)))-\mathrm{tr}(\phi'''(\Gamma_\gamma(h)))) 
< \delta_B,\quad h\in\mathcal H_{B, 2}.
\end{equation}
It also follows from \eqref{third-pert-lbd} that 
$$\frac{1}{N}\mathrm{Tr}(\phi_B'''(h))\geq\mathrm{tr}(\phi'''_B(h)) = \mathrm{tr}(\phi'''(\Gamma_\gamma(h)))  > \Delta(\hat h),\quad h\in\mathcal H_{B, 1},$$
and
$$\frac{1}{N}\mathrm{Tr}(\psi_B'''(h))\geq\mathrm{tr}(\psi'''_B(h)) = \mathrm{tr}(\psi'''(\Gamma_\gamma(h)))  > \Delta(\hat h),\quad h\in\mathcal H_{B, 1}.$$

Then, it follows from the inductive hypothesis that there are homomorphisms 
$\tilde{\phi}_B, \tilde{\psi}_B: B\to\Mat{N}$ such that
\begin{equation}\label{pert-B}
\norm{\tilde{\phi}_B(f_B)-\phi'''_B(f_B)} <\eps/2\quad \mathrm{and}\quad \norm{\tilde{\psi}_B(f_B)-\psi'''_B(f_B)} <\eps/2,\quad f\in\mathcal F,
\end{equation}
and there are decompositions
$$\tilde{\phi}_B=\tilde{\phi}_{B, 0}\oplus\underbrace{\tilde{\phi}_{B, 1}\oplus\cdots\oplus \tilde{\phi}_{B, 1}}_K\quad\textrm{and}\quad \tilde{\psi}_B=\tilde{\psi}_{B, 0}\oplus\underbrace{\tilde{\psi}_{B, 1}\oplus\cdots\oplus \tilde{\psi}_{B, 1}}_K$$
with $\tilde{\phi}_{B, 1}$ and $\tilde{\psi}_{B, 1}$ unitarily equivalent and
\begin{equation}\label{pre-ubd-tr-0}
\frac{1}{N}\mathrm{rank}(\tilde{\phi}_{B, 0}(a)) \leq \eta \cdot \frac{1}{N}\mathrm{rank}(\tilde{\phi}_{B}(a)), \quad a\in\mathcal F_B,
\end{equation}
\begin{equation}\label{pre-ubd-tr-1}
\frac{1}{N}\mathrm{rank}(\tilde{\psi}_{B, 0}(a)) \leq \eta \cdot  \frac{1}{N}\mathrm{rank}(\tilde{\psi}_{B}(a)), \quad a\in\mathcal F_B.
\end{equation}


By \eqref{rank-bd} and \eqref{pre-ubd-tr-0}, one has that for any $a\in\mathcal F_B$,
one has
\begin{equation}\label{ubd-tr-0}
\mathrm{tr}(\tilde{\phi}_{B, 0}(a))= \frac{1}{m}\mathrm{rank}(\tilde{\phi}_{B, 0}(a))=\frac{N}{m}\frac{1}{N}\mathrm{rank}(\tilde{\phi}_{B, 0}(a)) \leq\eta\cdot \frac{1}{m}\mathrm{rank}(\tilde{\phi}_{B}(a))=\eta\cdot \mathrm{tr}(\tilde{\phi}_{B}(a)),
\end{equation}
and for the same reason,
\begin{equation}\label{ubd-tr-1}
\mathrm{tr}(\tilde{\psi}_{B, 0}(a)) \leq \eta\cdot \mathrm{tr}(\tilde{\psi}_{B}(a)), \quad a\in\mathcal F_B.
\end{equation}


Consider the map $\bigoplus_{\norm{x_i}<1-\gamma}\pi_{x_i}$, there is a decomposition
$$\bigoplus_{\norm{x_i}<1-\gamma}\pi_{x_i}=\bigoplus_{\norm{x_i}<1-\gamma}\pi_{x'_i} \oplus \underbrace{(\bigoplus_{\norm{x_i}<1-\gamma}\pi _{x_i''})\oplus\cdots\oplus (\bigoplus_{\norm{x_i}<1-\gamma}\pi _{x_i''})}_K,$$
where each $x_i'$ and $x_i''$ are same as $x_i$ but with different multiplies, and $x_i'$ has multiplicity at most $K$.
Since each $x_i$ has multiplicity at least $M=\lfloor 2K/\eta\rfloor +1$, one has that
\begin{equation}\label{ubd-tr-1}
\mathrm{tr}(\bigoplus_{\norm{x_i}<1-\gamma}\pi_{x'_i}(a)) \leq \eta\cdot \mathrm{tr}( \bigoplus_{\norm{x_i}<1-\gamma}\pi_{x_i}(a)),\quad a\in A.
\end{equation}

Then set
$$\tilde{\phi}_0=\tilde{\phi}_{B, 0}\oplus\bigoplus_{\norm{x_i}<1-\gamma}\pi_{x'_i}  \quad\textrm{and}\quad \tilde{\psi}_0=\tilde{\psi}_{B, 0}\oplus \bigoplus_{\norm{x_i}<1-\gamma}\pi_{x'_i},$$
and
$$\tilde{\phi}_1=\tilde{\phi}_{B, 1}\oplus(\bigoplus_{\norm{x_i}<1-\gamma}\pi_{x''_i}) \quad\textrm{and}\quad \tilde{\psi}_1=\tilde{\psi}_{B, 1}\oplus(\bigoplus_{\norm{x_i}<1-\gamma}\pi_{x''_i}),$$
where $\tilde{\phi}_{B, 0}, \tilde{\phi}_{B, 1}, \tilde{\psi}_{B, 0}, \tilde{\psi}_{B, 1}$ are regarded as maps on $A$.

The homomorphisms
$$\tilde{\phi}=\tilde{\phi}_0\oplus \underbrace{\tilde{\phi}_{B, 1}\oplus\cdots\oplus \tilde{\phi}_{B, 1}}_K\quad\textrm{and}\quad \tilde{\psi}=\tilde{\psi}_0\oplus \underbrace{\tilde{\psi}_{B, 1}\oplus\cdots\oplus \tilde{\psi}_{B, 1}}$$
and these decompositions have the properties required in the statement of the theorem.

Indeed, by \eqref{first-pert-phi}, \eqref{first-pert-psi}, \eqref{sec-pert-F}, \eqref{third-pert-F}, and \eqref{pert-B}, one has
$$\norm{\phi(f)-\tilde{\phi}(f)}<\norm{\phi(f)- \phi'''(f)}+\norm{\phi'''(f)-\tilde{\phi}(f)}<\eps/2+\eps/2=\eps,\quad f\in\mathcal F,$$
and
$$\norm{\psi(f)-\tilde{\psi}(f)}<\norm{\psi(f)- \psi'''(f)}+\norm{\psi'''(f)-\tilde{\psi}(f)}<\eps/2+\eps/2=\eps,\quad f\in\mathcal F.$$
It is also clear that $\tilde{\phi}_1$ and $\tilde{\psi}_1$ are unitarily equivalent, and it follows from \eqref{ubd-tr-0},  \eqref{ubd-tr-1} and \eqref{ubd-tr-1} that the maps $\tilde{\phi}_0$ and $\tilde{\psi}_0$ satisfy the desired trace condition.
\end{proof}

Recall (from \cite{GLN-TAS})
\begin{lem}[Lemma 4.13 of \cite{GLN-TAS}]\label{stable-uniq}
Let $A$ be a unital separable nuclear residually finite dimensional C*-algebra satisfying the UCT, and let $\Delta: A^+_{{1}, q}\setminus\{0\} \to (0,1)$ be an order preserving map.  For any finite set ${\mathcal F}\subseteq A$ and any $\eps>0$, there exist
$\delta>0,$ a finite set ${\mathcal G}\subseteq A,$ a finite set ${\mathcal P}\subseteq \underline{K}(A),$ a finite set ${\mathcal H}\subseteq A_+^{1}\setminus \{0\}$, and an integer $K\geq 1$ satisfying the following condition:
For any two unital ${\mathcal G}$-$\delta$-multiplicative linear maps $\phi_1, \phi_2: A\to \Mat{n}$ (for some integer $n$) and any
unital homomorphism $\psi: A\to \Mat{m}$ with $m\geq n$ such that
\begin{enumerate}
\item $\tau\circ \psi(g)\geq \Delta(\hat{g})$, $g\in \mathcal H$,
\item $[\phi_1]|_{\mathcal P}=[\phi_2]|_{\mathcal P}$,
\end{enumerate}
there exists a unitary $U\in \Mat{Km+n}$ such that
\begin{equation*}
\norm{\phi_1(f)\oplus \underbrace{\psi(f)\oplus\cdots\oplus\psi(f)}_K-u^*(\phi_2(f)\oplus \underbrace{\psi(f)\oplus\cdots\oplus\psi(f)}_K)u}<\eps,\quad  f\in \mathcal F.
\end{equation*}
\end{lem}

%
%
%
%

\begin{thm}\label{unique-F}
Let $A$ be an NCCC. Let $\Delta: A^+_{1, q}\setminus\{0\} \to (0, 1)$ be an order preserving map. 
Let $(\mathcal F, \eps)$ be given. Then there are finite sets 
$\mathcal G\subseteq A$, 
$\mathcal H_1\subseteq A^+$, $\mathcal H_2\subseteq A^+$ and ${\mathcal P}\subseteq \underline{K}(A)$, and positive numbers $\delta, \sigma>0$ satisfying the following condition:

If $\phi, \psi: A\to \Mat{m}$ are unital $\mathcal G$-$\delta$-multiplicative maps such that 
\begin{enumerate}
\item $[\phi]|_{\mathcal P}=[\psi]|_{\mathcal P}$,
\item $\mathrm{tr}(\phi(h))>\Delta(\hat{h})$ and $\mathrm{tr}(\phi(h))>\Delta(\hat{h})$, $h\in\mathcal H_1$, and
\item $\abs{\mathrm{tr}(\phi)(h)-\mathrm{tr}(\psi)(h)}<\sigma$, $h\in\mathcal H_2$,
\end{enumerate}
then there is a unitary $u\in \Mat{m}$ such that
$$\norm{\phi(f)-u^*\psi(f)u}<\eps,\quad f\in\mathcal F.$$
\end{thm}
\begin{proof}
Applying Lemma \ref{stable-uniq} to the data $A$, $\Delta/4$, and $(\mathcal F, \eps/3)$, one obtains $(\mathcal G', \delta', \mathcal P, \mathcal H, K)$ satisfying the conditions of Lemma \ref{stable-uniq}. 

Applying Theorem \ref{decp-thm} to the data $A$, $\Delta/2$, $\mathcal F\cup\mathcal H$, $\min\{\eps/6,\Delta(\mathcal H)/4\}$ (in place of $\eps$), $\eta=1/2K,$ and $K$, one obtains $\mathcal H_1$, $\mathcal H_2$, and $2\sigma$ (in place of $\delta$) satisfying the conclusion of Theorem \ref{decp-thm}. Without loss of generality, one may assume that $\mathcal H_1$ and $\mathcal H_2$ are in the unital ball of $A$.

Applying Corollary 5.5 of \cite{GLN-TAS} to $$\mathcal F\cup\mathcal H_1\cup\mathcal H_2,\ \min\{\eps/6, \Delta(\mathcal H_1)/2, \Delta(\mathcal H_2)/2, \sigma/4\},\  \min\{1/2K, \Delta(\mathcal H_1)/2, \Delta(\mathcal H_2)/2, \sigma/4(1+\sigma)\},$$ one obtains $(\mathcal G, \delta)$. 

Then $\mathcal G$, $\delta$, $\sigma$, $\mathcal H_1$, $\mathcal H_2$, and $\mathcal P$ satisfy the conclusion of the theorem. Indeed, let $\phi, \psi: A\to \Mat{m}$ be $\mathcal G$-$\delta$-multiplicative maps such that 
\begin{enumerate}
\item $[\phi]|_{\mathcal P}=[\psi]|_{\mathcal P}$,
\item $\mathrm{tr}(\phi(h))>\Delta(\hat{h})$ and $\mathrm{tr}(\phi(h))>\Delta(\hat{h})$, $h\in\mathcal H_1$, and
\item $\abs{\mathrm{tr}(\phi)(h)-\mathrm{tr}(\psi)(h)}<\sigma$, $h\in\mathcal H_2$.
\end{enumerate}

By Theorem 5.5 of \cite{GLN-TAS}, there are $\phi_0, \phi_1, \psi_0, \psi_1: A\to \Mat{m}$ such that $\phi_0, \psi_0$ are $\mathcal G'$-$\delta'$-multiplicative, $\phi_1$, $\psi_1$ are homomorphisms, such that
\begin{equation}\label{pert-am-phi-1}
\norm{\phi(a)-\phi_0(a)\oplus\phi_1(a)} < \min\{\eps/6, \Delta(\mathcal H_1)/2, \Delta(\mathcal H_2)/2, \sigma/4\}, \quad a\in \mathcal F\cup\mathcal H_1\cup\mathcal H_2,
\end{equation}
\begin{equation}\label{pert-am-psi-1}
\norm{\psi(a)-\psi_0(a)\oplus\psi_1(a)} < \min\{\eps/6, \Delta(\mathcal H_1)/2, \Delta(\mathcal H_2)/2, \sigma/4\},  \quad   a\in \mathcal F\cup\mathcal H_1\cup\mathcal H_2,
\end{equation}
and
$$\mathrm{tr}(\phi_0(1))=\mathrm{tr}(\psi_0(1)) <  \min\{1/2K, \Delta(\mathcal H_1)/2, \Delta(\mathcal H_2)/2, \sigma/4(1+\sigma)\}.$$

Consider the unital homomorphisms $\phi_1, \psi_1: A\to p\Mat{m}p$, where $p=\phi_1(1)=\psi_1(1)$. One has that 
\begin{enumerate}
\item $\tau(\phi_1(h)), \tau(\phi_1(h))>\Delta(\hat{h})/2,\quad h\in \mathcal H_1$, and
\item $\abs{\tau(\phi_1(h))-\tau(\psi_1(h))}<2\sigma,\quad h\in\mathcal H_2$,
\end{enumerate}
where $\tau\in\mathrm{T}(p\Mat{m}p)$.

By Theorem \ref{decp-thm}, up to unitary equivalence, there are homomorphisms $\phi_1', \psi_1', \mu: A\to p\Mat{m}p$ such that
\begin{equation}\label{pert-am-phi-2}
\norm{\phi_1(a) - \phi_1'(a)\oplus \underbrace{\mu(a)\oplus\cdots\oplus\mu(a)}_K}<\min\{\eps/6, \Delta(\mathcal H)/4\},\quad a\in\mathcal F\cup\mathcal H,
\end{equation}
\begin{equation}\label{pert-am-psi-2}
\norm{\psi_1(a) - \psi_1'(a)\oplus \underbrace{\mu(a)\oplus\cdots\oplus\mu(a)}_K}<\min\{\eps/6, \Delta(\mathcal H)/4\},\quad a\in\mathcal F\cup\mathcal H,
\end{equation}
and
$$\tau(\phi_1'(a)) \leq  \frac{1}{2K}  \cdot \tau(\phi_1(a)) \quad\mathrm{and}\quad \tau(\psi_1'(a)) \leq  \frac{1}{2K} \cdot \tau(\psi_1(a)),\quad a\in\mathcal F\cup\mathcal H, \tau\in\mathrm{T}(p\Mat{m}p).$$
Therefore, one has that 
$$\tau(\mu(h)) > \frac{1}{4}\Delta(\hat{h}),\quad h\in\mathcal H,\ \tau\in\mathrm{T}(q\Mat{m}q),$$
where $q=\mu(1)$.
Consider the map 
$$(\phi_0\oplus \phi_1')\oplus (\underbrace{\mu \oplus\cdots\oplus\mu }_K)\quad
\mathrm{and}
\quad (\psi_0\oplus \psi_1')\oplus (\underbrace{\mu \oplus\cdots\oplus\mu }_K),$$
and note that
$$\mathrm{tr}(\phi_0(1)\oplus\phi'_1(1))=\mathrm{tr}(\psi_0(1)\oplus\psi_1'(1))<\mathrm{tr}(\mu(1)).$$
It then follows from Lemma \ref{stable-uniq} that there is a unitary $u\in \Mat{m}$ such that for any $a\in\mathcal F$,
$$\norm{(\phi_0(a)\oplus \phi_1'(a))\oplus (\underbrace{\mu(a) \oplus\cdots\oplus\mu(a) }_K) - u^*((\psi_0(a)\oplus \psi_1'(a))\oplus (\underbrace{\mu(a) \oplus\cdots\oplus\mu(a) }_K))u} < \frac{\eps}{3}.$$
It then follows from \eqref{pert-am-phi-1}, \eqref{pert-am-psi-1}, \eqref{pert-am-phi-2}, and \eqref{pert-am-psi-2} that 
$$\norm{\phi(a) - u^* \psi(a)u} <\eps,\quad a\in\mathcal F,$$
as desired.
\end{proof}

Note that $\mathrm{KL}(A, Q)\cong\mathrm{Hom}(\Kzero(A), \Kzero(Q))$, a straightforward consequence is
\begin{cor}\label{unique-Q}
Let $A$ be an NCCC. Let $\Delta: A^+_{1, q}\setminus\{0\} \to (0, 1)$ be an order preserving map. 
Let $(\mathcal F, \eps)$ be given. Then there are finite sets 
$\mathcal H_1\subseteq A^+$ and $\mathcal H_2\subseteq A^+$ and a positive number $\delta>0$ satisfying the following condition:

If $\phi, \psi: A\to Q$ are unital homomorphisms such that 
\begin{enumerate}
\item $[\phi]_0=[\psi]_0$,
\item $\mathrm{tr}(\phi(h))>\Delta(\hat{h})$ and $\mathrm{tr}(\phi(h))>\Delta(\hat{h})$, $h\in\mathcal H_1$, and
\item $\abs{\mathrm{tr}(\phi)(h)-\mathrm{tr}(\psi)(h)}<\delta$, $h\in\mathcal H_2$,
\end{enumerate}
then there is a unitary $u\in Q$ such that
$$\norm{\phi(f)-u^*\psi(f)u}<\eps,\quad f\in\mathcal F.$$
\end{cor}

It is also worth pointing out the following corollary:
\begin{cor}
Let $A$ be a subhomogeneous C*-algebra. Let $\Delta: A^+_{1, q}\setminus\{0\} \to (0, 1)$ be an order preserving map. 
Let $(\mathcal F, \eps)$ be given. Then there are finite sets 
$\mathcal G\subseteq A$, 
$\mathcal H_1\subseteq A^+$, $\mathcal H_2\subseteq A^+$ and ${\mathcal P}\subseteq \underline{K}(A)$, and positive numbers $\delta, \sigma>0$ satisfying the following condition:

If $\phi, \psi: A\to \Mat{m}$ are unital $\mathcal G$-$\delta$-multiplicative maps such that 
\begin{enumerate}
\item $[\phi]|_{\mathcal P}=[\psi]|_{\mathcal P}$,
\item $\mathrm{tr}(\phi(h))>\Delta(\hat{h})$ and $\mathrm{tr}(\phi(h))>\Delta(\hat{h})$, $h\in\mathcal H_1$, and
\item $\abs{\mathrm{tr}(\phi)(h)-\mathrm{tr}(\psi)(h)}<\sigma$, $h\in\mathcal H_2$,
\end{enumerate}
then there is a unitary $u\in \Mat{m}$ such that
$$\norm{\phi(f)-u^*\psi(f)u}<\eps,\quad f\in\mathcal F.$$
\end{cor}
\begin{proof}
This follows the fact that any subhomogeneous C*-algebra can be locally approximated by NCCCs (Theorem 2.15 of \cite{ENST-ASH}). 
\end{proof}

\section{Tracial factorization and tracial approximation}

Recall that 
\begin{defn}[\cite{LinTAF1}, \cite{EN-TApprox}]
Let $\mathcal S$ be a class of unital C*-algebras. A C*-algebra $A$ is said to be tracially approximated by the C*-algebras in $\mathcal S$, and one writes $A\in \mathrm{TA}\mathcal S$, if the following condition holds: For any finite set $\mathcal F\subseteq A$, any $\eps>0$, and any nonzero $a\in A^+$, there is a nonzero sub-C*-algebra $S\subseteq A$ such that $S\in\mathcal S$,  and if $p=1_S$, then
\begin{enumerate}
\item $\norm{pf-fp}<\eps$, $f\in\mathcal F$,
\item $pfp\in_\eps S$, $f\in\mathcal F$, and
\item $1-p$ is Murray-von Neumann equivalent to a subprojection of $\overline{aAa}$.
\end{enumerate}
\end{defn}

One particularly important class $S$ of building blocks is the class of Elliott-Thomsen algebras.
\begin{defn}(\cite{ET-PL}, \cite{point-line})
A C*-algebra $C$ is said to be an Elliott-Thomsen algebra if 
$$C=\{(a, f)\in E\oplus (F\otimes\mathrm{C}([0, 1])): f(0)=\varrho_0(a), f(1)=\varrho_1(a)\}$$
for some finite dimensional C*-algebras $E$, $F$, where $\varrho_0, \varrho_1: E\to F$ are unital homomorphisms. 
Denote by $\pi_\infty$ the standard quotient map
$$\pi_\infty: A\ni (a, f)\mapsto a\in E.$$

Let us denote the class of unital Elliott-Thomsen algebras by $\mathcal C$, and denote the class of unital Elliott-Thomsen algebras with trivial $\Kone$-group by $\mathcal C_0$.
\end{defn}

\begin{rem}
In fact, by Corollary 29.3 of \cite{GLN-TAS}, one has $\mathrm{TA}\mathcal C=\mathrm{TA}\mathcal C_0$.
\end{rem}

\begin{rem}
In fact, the class of unital Elliott-Thomsen algebras is exactly the class of NCCCs with dimensions of cells at most one; see \cite{ENST-ASH}. 
\end{rem}

For $\mathrm{TA}\mathcal C_0$ algebras, one has the following classification theorem.
\begin{thm}[Corollary 28.7 of \cite{GLN-TAS}]
Let $A, B$ be unital separable amenable simple C*-algebras satisfying the UCT. Assume that $A$, $B$ are Jiang-Su stable, and assume that $A\otimes Q\in\mathrm{TA}\mathcal C_0$ and $B\otimes Q\in \mathrm{TA}\mathcal C_0$. Then $A\cong B$ if and only if $\mathrm{Ell}(A)\cong\mathrm{Ell}(B)$.
\end{thm}

In this section, let us show that for any separable simple unital locally ASH algebra $A$, one has that $A\otimes Q\in\mathrm{TA}\mathcal C_0$. 

\begin{thm}\label{TFT}
Let $A$ be a unital simple separable locally approximately subhomogeneous (ASH) C*-algebra satisfying $A\cong A\otimes Q$. Then, for any finite set $\mathcal F\subseteq A$ and any $\eps>0$, there exist an Elliott-Thomsen algebra $C$ with $\Kone(C)=\{0\}$, a unital completely positive linear map $\Phi: A\to C$, and a unital embedding $\Psi: C\to A$ such that
\begin{enumerate}
\item $\Phi$ is $\mathcal F$-$\eps$-multiplicative, and
\item $\abs{\tau(f)-\tau(\Psi(\Phi(f)))}<\eps$, $f\in\mathcal F$, $\tau\in \tr(A)$.
\end{enumerate}
\end{thm}

\begin{proof}
Without loss of generality, one may assume that $1\in\mathcal F$ and each element of $\mathcal F$ is self-adjoint and has norm at most $1$. 

Let $A$ be a unital separable simple locally ASH algebra satisfying $A\cong A\otimes Q$. By Theorem 2.15 of \cite{ENST-ASH}, the C*-algebra $A$ can be locally approximated by unital NCCCs. Therefore, without loss of generality, one may assume that there is a sub-C*-algebra $A_1\subseteq A$ such that $A_1$ is a NCCC and
$\mathcal F\subseteq A_1.$

Put
$$G_{A_1}=\rho_{A_1}(\Kzero(A_1)),\ G^+_{A_1}=\rho_{A_1}(\Kzero(A_1))\cap\aff^+(\tr(A_1)),\ u_{A_1}=\rho_{A_1}([1]),$$
and fix a set $\{p_1, ..., p_n\}$ which generates the group $G_{A_1}$. Note that by Lemma \ref{fg-quotient}, the positive cone $G^+_{A_1}$ is finitely generated.

For each $h\in A_+$, define 
$$\Delta(h)=\inf\{\tau(\iota(h)): \tau\in \tr(A)\},$$
where $\iota: A_1\to A$ is the embedding map. Since $A$ is simple, the map $\Delta$ induces a order preserving map from $A_+^{1, q}$ to $(0, 1)$. Let us still denote this by $\Delta$.

Let $\mathcal H_1\subseteq (A_1)_+$, $\mathcal H_2\subseteq (A_1)_+$ and $\delta_0>0$ be the finite sets and constant of Corollary \ref{unique-Q} with respect to $\mathcal F\cdot\mathcal F$, $\eps/4$, and $\Delta/4$.

Let $\mathcal H$ be the subset of Theorem \ref{est-AM} with respect to $A_1$, $\mathcal F\cup\mathcal H_1\cup\mathcal H_2$ (in the place of $\mathcal F$), $\min\{\eps/4, \delta_0/4, \Delta(\hat{h})/4: h\in\mathcal H_1\}$ (in the place of $\eps$). Put
$$\sigma=\frac{1}{2}\min\{\Delta(\hat{h}): h\in\mathcal H\},$$
and denote by $\delta_1$ the constant of Theorem \ref{est-AM} with respect to $\sigma$. 

Put
$$G=\rho_A(\Kzero(A)),\ G^+=\rho_A(\Kzero(A))\cap\aff^+(\tr(A)),\ u=\rho_A([1]).$$
Then $(G, G^+, u)$ is a unperforated order-unit group, and there is a natural pairing between $(G, G^+, u)$ and $\tr(A)$ induced by $\rho_A$. Still denote this pairing map by $\rho_A$.
Note that one has the following commutative diagram:
\begin{displaymath}
\xymatrix{
\Kzero(A_1) \ar[r]^{[\iota]_0} \ar[d]^{\rho_{A_1}} & \Kzero(A) \ar[d]_{\rho_A}\\
G_{A_1} \ar[r]^{(\iota)_*} & G.
}
\end{displaymath}

Consider $((G, G^+, u), \tr(A), \rho_A)$. By Theorem 5.2.2 of \cite{point-line}, there is an inductive system $C'=\varinjlim(C'_i, \eta'_i)$ with $C'_i\in\mathcal C_0$ and $\eta_i$ injective such that there is an isomorphism
$$\Xi: ((G, G^+, u), \tr(A), \rho_A) \to ((\Kzero(C'), \Kzero(C'), [1]), \tr(C'), \rho_C').$$
Consider $C=C'\otimes Q=\varinjlim(C_i, \eta_i)$, where $C_i=C'_i\otimes Q$ and $\eta_i=\eta'_i\otimes\mathrm{id}_Q$. One has
$$((G, G^+, u), \tr(A), \rho_A) \cong ((\Kzero(C), \Kzero(C), [1]), \tr(C), \rho_C),$$
and let us still denote by $\Xi$ the isomorphism. In the remaining part of the paper, let us also use $\Xi$ to denote its restriction to $G$ or to $\aff(\tr(A))$, depending on the context.

Consider the following diagram:
\begin{displaymath}
\xymatrix{
G_{A_1}\ar[rrr]^{[\iota]_0}& & & G \ar[d]^\Xi \\
\Kzero(C_1) \ar[r]_-{[\eta_1]_0} & \Kzero(C_2) \ar[r]_-{[\eta_2]_0} & \cdots \ar[r] & \Kzero(C).
}
\end{displaymath}
Since the positive cone of $G_{A_1}$ is finitely generated (Lemma \ref{fg-quotient}), the positive map $[\iota]_0$ can be lifted to a positive homomorphism $G_{A_1}\to \Kzero(C_n)$ for sufficiently large $n$. Without loss of generality, one may assume that $[\iota]_0$ has a lifting $\kappa: G_{A_1}\to \Kzero(C_1)$, making the diagram commutative:
\begin{displaymath}
\xymatrix{
G_{A_1}\ar[rrr]^{[\iota]_0}\ar[d]_\kappa & & & G \ar[d]^\Xi \\
\Kzero(C_1) \ar[r]_-{[\eta_1]_0} & \Kzero(C_2) \ar[r]_-{[\eta_2]_0} & \cdots \ar[r] & \Kzero(C).
}
\end{displaymath}

By Lemma \ref{approx-finite-tr}, after a telescoping of the inductive system $(C_i, \eta_i)$, there is also an approximate lifting, making the diagram of affine function spaces,
\begin{displaymath}
\xymatrix{
\aff(\tr(A_1))\ar[rrr]^{(\iota)_*}\ar[d]^\gamma & & & \aff(\tr(A))) \ar[d]^\Xi \\
\aff(\tr(C_1))) \ar[r]_-{(\eta_1)_*} & \aff(\tr(C_2))) \ar[r]_-{(\eta_2)_*} & \cdots \ar[r] & \aff(\tr(C))),
}
\end{displaymath}
approximately commutative, and such that
\begin{equation}\label{lift-cpt}
\abs{\tau(\kappa([p_i]))-\gamma([p_i])(\tau)}<\delta_1,\quad \tau\in\tr(C_1),\ 1\leq i\leq n,
\end{equation}
\begin{equation}\label{lift-lbd}
\gamma(\hat{h})(\tau) > \frac{1}{2}\Delta(h)>\sigma,\quad h\in\mathcal H\cup\mathcal H_1,\ \tau\in\tr(C_1),
\end{equation}
and
\begin{equation}\label{lift-tr}
\abs{(\Xi^{-1}\circ (\eta_{1, \infty})_*\circ\gamma)(\hat{f})(\tau)-(\iota)_*(\hat{f})(\tau)} <\eps/8,\quad \tau\in\tr(A),\ f\in \mathcal F.
\end{equation}

Write 
$$C_1=C_1'\otimes Q=\{(a, f)\in E\oplus (F\otimes\mathrm{C}([0, 1])): f(0)=\varrho_0(a), f(1)=\varrho_1(a)\},$$
where 
$$E=\bigoplus_{i=1}^p Q,\quad F=\bigoplus_{j=1}^l Q,$$
for natural numbers $p, l$, and $\varrho_0, \varrho_1: E\to F$ are unital homomorphisms.

On each interval $[0, 1]_j$, choose a partition
$$0=t_0<t_1<t_2<\cdots<t_{k-1}< t_k = 1$$
such that 
\begin{equation}\label{dense-part}
\abs{\gamma(\hat{f})(\tau_s) - \gamma(\hat{f})(\tau_{t_i})} < \min\{\eps/4, \delta_0/2\},\quad s\in[t_i, t_{i+1}], f\in\mathcal F\cup\mathcal H_2,
\end{equation}
where $\tau_s=\mathrm{tr}\circ\pi_s$ and $\mathrm{tr}$ is the canonical trace of $Q$. One may assume that $k$ is sufficiently large that $$2\pi/(k-1)<\eps/8.$$

For each $0< i < k$, define
$$\kappa_i=[\pi_{t_i}]_0\circ\kappa: \Kzero(A_1) \to \Kzero(Q)\cong\Ratn,$$
and
$$\gamma_i=(\pi_{t_i})_*\circ\gamma: \aff(\tr(A_1)) \to \aff(\tr(Q))\cong\Real.$$
By \eqref{lift-cpt}, each pair $(\kappa_i, \gamma_i)$ is $\delta_1$-compatible on $[p_i]$, $1\leq i\leq n$. By \eqref{lift-lbd}, one has that 
$$\gamma_i(\hat{h})(\mathrm{tr})>\sigma,\quad h\in\mathcal H.$$
Therefore, by Theorem \ref{est-AM}, there is a homomorphism $\phi_i: A_1\to Q$ such that
$$[\phi_i]_0=\kappa_i$$
and
\begin{equation}\label{small-pert-est}
\abs{\gamma_i(\hat{h})(\mathrm{tr})-\mathrm{tr}(\phi_i(h))} < \min\{\eps/4, \delta_0/4, \Delta(\hat{h})/4: h\in\mathcal H_1\},\quad h\in\mathcal F\cup\mathcal H_1\cup\mathcal H_2.
\end{equation}
Together with \eqref{lift-lbd}, it then follows that  
\begin{equation}\label{lbd-trace-final}
\mathrm{tr}(\phi_i(h))> \gamma_i(\hat{h})(\mathrm{tr})-\frac{1}{4}\Delta(\hat{h})>\frac{1}{4}\Delta(\hat{h}),\quad h\in\mathcal H_1.
\end{equation}
It also follows from \eqref{small-pert-est} and \eqref{dense-part} that for any $h\in \mathcal H_2$ and any $1\leq i \leq k-2$, 
\begin{eqnarray}\label{trace-close}
\abs{\mathrm{tr}(\phi_i(h))-\mathrm{tr}(\phi_{i+1}(h))} & \leq & \abs{\mathrm{tr}(\phi_i(h))-\gamma_i(\hat{h})(\mathrm{tr})}+\abs{\gamma_i(\hat{h})(\mathrm{tr})-\gamma_{i+1}(\hat{h})(\mathrm{tr})}  \\
&& + \abs{\mathrm{tr}(\phi_{i+1}(h))-\gamma_{i+1}(\hat{h})(\mathrm{tr})} \nonumber \\
&<& \delta_0/4 + \delta_0/2 +\delta_0/4 =\delta_0. \nonumber
\end{eqnarray}
Since $\pi_i$ is homotopic to $\pi_{i+1}$, it is clear that $[\kappa_i]=[\kappa_{i+1}]$, $1\leq i\leq k-2$, and therefore
\begin{equation}\label{same-k}
[\phi]_0=[\phi_{i+1}]_0,\quad 1\leq i\leq k-2.
\end{equation}

Denote by $\pi_\infty: C_1\to E$ the standard quotient map, and consider 
$$\kappa_\infty=[\pi_\infty]_0\circ\kappa: \Kzero(A_1)\to\Kzero(E)\quad \mathrm{and}\quad \gamma_\infty=[\pi_\infty]\circ\gamma: \aff(\tr(A_1))\to\aff(\tr(E)).$$
The same argument as above shows that there is a homomorphism $\phi_\infty: A_1\to E$ such that 
\begin{equation}\label{same-k-bdary}
[\phi_\infty]_0=\kappa_\infty,
\end{equation}
and
\begin{equation}\label{small-pert-bdary}
\abs{\gamma_\infty(\hat{h})(\tau) - \tau(\phi_\infty(h))}< \min\{\eps/4, \delta_0/4, \Delta(\hat{h})/4: h\in\mathcal H_1\},\quad h\in\mathcal H_1\cup\mathcal H_2,\ \tau\in\tr(E).
\end{equation}

Define $$\phi_0=\varrho_0\circ\phi_\infty \quad\mathrm{and}\quad \phi_k=\varrho_1\circ\phi_\infty $$
and consider the restrictions of these maps to the $j$-th direct summand; still denote them by $\phi_0$ and $\phi_k$ respectively. It then follows from \eqref{same-k-bdary} that
\begin{equation}\label{same-k-all}
[\phi_0]_0=[\phi_k]=[\phi_i]_0,\quad 1\leq i\leq k-1,
\end{equation}
and it follows from \eqref{small-pert-bdary} and \eqref{lift-lbd} that
\begin{equation}\label{lbd-trace-bdary}
\mathrm{tr}(\phi_0(h)), \mathrm{tr}(\phi_k(h)) > \frac{1}{4}\Delta(\hat{h}),\quad h\in\mathcal H_1.
\end{equation}
Moreover, with \eqref{small-pert-bdary} and \eqref{dense-part}, the same argument as that of \eqref{trace-close} shows that
\begin{equation}\label{trace-close-bdary}
\abs{\mathrm{tr}(\phi_0(h))-\mathrm{tr}(\phi_1(h))}<\delta_0\quad\mathrm{and}\quad \abs{\mathrm{tr}(\phi_{k-1}(h))-\mathrm{tr}(\phi_k(h))}<\delta_0,\quad h\in \mathcal H_2.
\end{equation}

Thus, with \eqref{same-k-all}, \eqref{lbd-trace-final}, \eqref{lbd-trace-bdary}, \eqref{trace-close} and \eqref{trace-close-bdary}, Corollary \ref{unique-Q} implies that there are unitaries
$u_1, u_2, ..., u_{k-1}\in Q$
such that
$$
\norm{\phi_i(f)-u_{i+1}^*\phi_{i+1}(f)u_{i+1}}<\eps/4,\quad f\in\mathcal F\cdot\mathcal F,\ 0\leq i\leq k-2.
$$
Define $v_0=1$, and
$$v_i=u_iu_{i-1} \cdots u_1,\quad i=1, ..., k-1.$$
Then, for any $0\leq i\leq k-2$ and any $f\in\mathcal F\cdot\mathcal F$, one has
\begin{eqnarray*}
&&\norm{\mathrm{Ad}(v_i)\circ\phi_i(f)-\mathrm{Ad}(v_{i+1})\circ\phi_{i+1}(f)}\\
&=& \norm{(u_i \cdots u_1)^*\phi_i(f)(u_i \cdots u_1)-(u_{i+1} \cdots u_1)^*\phi_{i+1}(f)(u_{i+1} \cdots u_1)}\\
&=& \norm{\phi_i(f)-u_{i+1}^*\phi_{i+1}(f)u_{i+1}}<\eps/4.
\end{eqnarray*}
Replacing each homomorphism $\phi_i$ by $\mathrm{Ad}(v_i)\circ\phi_i$ for $i=1, ..., k-1$, and still denoting it by $\phi_i$, one has
$$
\norm{\phi_i(f)-\phi_{i+1}(f)}<\eps/4,\quad f\in\mathcal F\cdot\mathcal F,\ 0\leq i\leq k-2.
$$
Note that the replacement of $\phi_i$ does not change the induced map on the invariant, and hence one still has 
$$[\phi_{k-1}]_0= [\phi_k]_0,$$
$$\mathrm{tr}(\phi_0(h)), \mathrm{tr}(\phi_k(h)) > \frac{1}{4}\Delta(\hat{h}),\quad h\in\mathcal H_1,$$
and
$$\abs{\mathrm{tr}(\phi_{k-1}(h))-\mathrm{tr}(\psi_k(h))}<\delta_0,\quad h\in \mathcal H_2.$$

Applying Corollary \ref{unique-Q} again, we obtain a unitary $w\in Q$ such that 
$$\norm{w^*\phi_{k-1}(f)w-\phi_k(f)}<\eps/4,\quad f\in\mathcal F\cdot\mathcal F.$$
Since $Q$ is AF, and any unitary in a finite dimensional C*-algebra can be connected to the identity along a path with length at most $\pi$ (i.e., it has exponential length at most $\pi$), there are unitaries
$$1=w_0, w_1, ..., w_{k-2}, w_{k-1}=w \in Q$$
such that 
$$\norm{w_{i}-w_{i-1}}< 2\pi/(k-1)<\eps/8.$$
Hence,
$$\norm{w^*_i \phi_i(f)w_i-w^*_{i+1}\phi_{i+1}(f)w_{i+1}}<3\eps/8,\quad f\in\mathcal F\cdot\mathcal F, 0\leq i\leq k-2.$$
Replace each homomorphism $\phi_i$ again by $\mathrm{Ad}(w_i)\circ\phi_i$, $1\leq i \leq k-1$, and still denote it by $\phi_i$. One then has 
\begin{equation}\label{app-div}
\norm{\phi_i(f)-\phi_{i+1}(f)}<3\eps/8,\quad f\in\mathcal F\cdot\mathcal F,\ 0\leq i\leq k-1.
\end{equation}

Define the positive linear map $\phi: A_1\to \mathrm{C}([0, 1]_j, Q)$ by
\begin{displaymath}
\Phi_j(f)(t)=\frac{t_{i+1}-t}{t_{i+1}-t_{i}}\phi_i(f) + \frac{t-t_i}{t_{i+1}-t_i}\phi_{i+1}(f),\quad \textrm{if $t\in[t_i, t_{i+1}]$}.
\end{displaymath}
By \eqref{app-div}, the map $\Phi_j$ is $\mathcal F$-$\eps$-multiplicative. It also follows from \eqref{dense-part} that if $t\in[t_i, t_{i+1}]$, then, for any $f\in\mathcal F$,
\begin{eqnarray}\label{app-aff-ind}
&& \abs{\gamma(\hat{f})(\tau_t)-\tau_t(\Phi_j(f))} \\
& \leq& \abs{\gamma(\hat{f})(\tau_{t_i})-\tau_t(\Phi_j(f))} + \eps/4\quad \textrm{(by \eqref{dense-part})}  \nonumber \\
&=& \abs{\gamma_i(\hat{f})(\mathrm{tr})  -(\frac{t_{i+1}-t}{t_{i+1}-t_{i}}\mathrm{tr}(\phi_i(f)) + \frac{t-t_i}{t_{i+1}-t_i}\mathrm{tr}(\phi_{i+1}(f)))}+\eps/4 \nonumber \\
&\leq &\abs{\gamma_i(\hat{f})(\mathrm{tr}) -\mathrm{tr}(\phi_i(f))} +5\eps/8\quad \textrm{(by \eqref{app-div})}  \nonumber \\
&< & \eps/4+5\eps/8=7\eps/8\quad \textrm{(by \eqref{small-pert-est})}. \nonumber
\end{eqnarray}

Repeat this construction for all $j=1, ..., l$, and note that the maps $\Phi_1, \Phi_2, ..., \Phi_l$ induce a map $\Phi: A_1\to C_1$. Since each $\Phi_j$ is $\mathcal F$-$\eps$ multiplicative, the map $\Phi$ is $\mathcal F$-$\eps$-multiplicative. By \eqref{app-aff-ind}, one has
\begin{equation}\label{app-tr-0}
\abs{\gamma(\hat{f})(\tau)-\tau(\Phi(f))}<7\eps/8,\quad f\in\mathcal F, \tau\in\tr(C_1).
\end{equation}

Now, let us construct an embedding $\Psi: C_1\to A$ such that
$$|\tau(f)-\tau(\Psi(\Phi(f)))|<\eps,\quad f\in\mathcal H, \tau\in\tr(A).$$

Since $A\cong A\otimes Q$, the subgroup $H:=\ker\rho_A\subseteq \Kzero(A)$ is divisible, and therefore
the exact sequence
\begin{displaymath}
\xymatrix{
0\ar[r] & H \ar[r] & \Kzero(A) \ar[r]^-{\rho_A} & G \ar[r] & 0
}
\end{displaymath}
splits. 
Pick a decomposition
$$\Kzero(A)=G\oplus H.$$ Since $\Kzero(A)$ is weakly unperforated, the order on $(G\oplus H)$ is completely determined by that on $G$.

Define $$\kappa': G \ni g\mapsto (g, 0) \in G\oplus H=\Kzero(A).$$
Then $\kappa'$ is a positive homomorphism, and the pair
$$(\kappa'\circ(\Xi^{-1}|_{\Kzero(C)}), \Xi|_{\aff(\tr(C))})$$
is compatible. It induces a positive homomorphism
$$\theta: \mathrm{Cu}^{\sim}(C) \to \mathrm{Cu}^{\sim}(A).$$
By Theorem 1 of \cite{Robert-Cu}, there is a homomorphism $\psi: C\to A$ such that the Cuntz map induced by $\psi$ is $\theta$. In particular, one has that
\begin{equation}\label{partial-lift}
(\psi)_*=\Xi^{-1}|_{\aff(\tr(C))}.
\end{equation}

Then the map
$$\Psi=\psi\circ\eta_{1, \infty}$$
satisfies the conclusion of the theorem (together with $\Phi$).
Indeed, since $C$ is simple, the map $\psi$ is an embedding, and therefore $\Psi$ is an embedding.
Moreover, for any $f\in \mathcal F$, one has 
\begin{eqnarray*}
&& \abs{\tau(\iota(f))-\tau(\Psi\circ\Phi(f))} \\
& = & \abs{(\iota)_*(\hat{f})(\tau)-(\Psi)_*((\Phi)_*(\hat{f}))(\tau)} \\
&=&  \abs{(\iota)_*(\hat{f})(\tau)-(\psi)_*((\eta_{i, \infty}\circ\Phi)_*(\hat{f}))(\tau)} \\
& = & \abs{(\iota)_*(\hat{f})(\tau)-(\Xi^{-1}\circ(\eta_{i, \infty})_*\circ(\Phi)_*)(\hat{f})(\tau)}\quad \textrm{(by \eqref{partial-lift})}\\
&< & \abs{(\iota)_*(\hat{f})(\tau)-(\Xi^{-1}\circ(\eta_{i, \infty})_*\circ\gamma)(\hat{f})(\tau)} + 7\eps/8\quad \textrm{(by \eqref{app-tr-0})} \\
&\leq & \eps/8+ 7\eps/8=\eps\quad \textrm{(by \eqref{lift-tr})},
\end{eqnarray*}
as desired.
\end{proof}

\begin{rem}
With a slight modification (a perturbation of the linear map $\gamma$), the same argument as Theorem \ref{TFT} shows that the same statement holds for C*-algebras which are tracially approximated by subhomogeneous C*-algebras.
\end{rem}

The passage from Theorem \ref{TFT} to the actual tracial approximation is an application of the following very important theorem due to Winter:
\begin{thm}[Theorem 2.2 of \cite{Winter-TA}]\label{WTA}
Let $\mathcal S$ be a class of separable unital C*-algebras which have a finite presentation with weakly stable relations. Suppose further that $\mathcal S$ is closed under taking direct sums and under taking tensor products with finite dimensional C*-algebras, and that $\mathcal S$ contains all finite dimensional C*-algebras.

Let $A$ be a separable, simple, unital C*-algebra with $\mathrm{dim}_{\mathrm{nuc}} A<\infty$ and $\tr(A)\neq\O$, and let
$$(
\xymatrix{
A\ar[r]^{\sigma_i} & B_i \ar[r]^{\varrho_i}& A 
}
)_{i\in\mathbb N}$$
be a system of maps with the following properties:
\begin{enumerate}
\item $B_i\in\mathcal S$, $i\in\mathbb N$,
\item $\varrho_i$ is an embedding for each $i\in\mathbb N$,
\item $\sigma_i$ is a completely positive contraction for each $i\in\mathbb N$,
\item $\bar{\sigma}: A\to \prod_{i\in\mathbb N}B_i/\bigoplus_{i\in\mathbb N}B_i$ induced by $\sigma_i$ is a unital homomorphism
\item $\sup\{\abs{\tau(\varrho_i\sigma_i(a)-a)}: \tau\in\tr(A)\}\to 0$, as $i\to\infty$ for each $a\in A$.
\end{enumerate}
Then $A\otimes Q\in\mathrm{TA}\mathcal S$.
\end{thm}

With this and Theorem \ref{TFT}, one has
\begin{thm}\label{classification}
Let $A$ be a unital simple separable locally ASH C*-algebra. Then $A\otimes Q\in\mathrm{TA}\mathcal C_0$, where $\mathcal C_0$ is the class of unital Elliott-Thomsen algebras with trivial $\Kone$-group. In particular, if $A\cong A\otimes\mathcal Z$, where $\mathcal Z$ is the Jiang-Su algebra, then $A$ is classifiable (by means of the naive Elliott invariant). (The converse is also true.)
\end{thm}

\begin{proof}
By Theorem 3.1 of \cite{ENST-ASH}, one has $\mathrm{dr}(A\otimes Q)\leq 2$, and in particular, $\mathrm{dim}_{\mathrm{nuc}} (A\otimes Q)\leq 2<+\infty$. It then follows from Theorems \ref{TFT} and \ref{WTA} that $A\otimes Q\in\mathrm{TA}\mathcal C_0$. By the classification theorem of \cite{GLN-TAS} (based in particular on the deformation technique of \cite{Winter-Z} and \cite{Lin-App}---see also \cite{L-N}), the C*-algebra $A$ is classifiable. 
\end{proof}

\begin{cor}
Let $A$ be a simple separable unital locally ASH (respectively, locally AH) algebra. Then $A\otimes\mathcal Z$ is an ASH (respectively, AH) algebra. 
\end{cor}
\begin{proof}
By Theorem \ref{classification}, the C*-algebra $A\otimes\mathcal Z$ is classifiable by means of the Elliott invariant. By \cite{point-line} and \cite{GJS-Z}, the Elliott invariant for separable, Jiang-Su stable, simple, unital, finite C*-algebras (in particular, locally ASH algebras) is exhausted by ASH algebras (by Theorem 3 of \cite{GJS-Z}, finiteness implies stable finiteness in this setting). Furthermore, by \cite{Vill-EHS}, the Elliott invariant for separable, Jiang-Su stable, simple, unital, locally AH algebras is exhausted by AH algebras. (In both settings, the models have no dimension growth.)
\end{proof}

The classification of locally ASH algebras (Theorem \ref{classification}) in fact allow us to recover the recent classification result for the C*-algebra of a minimal homeomorphism---assumed to have mean dimension zero but not to be uniquely ergodic (\cite{Karen-sphere}, \cite{Lin-Dym}---the uniquely ergodic case was dealt with in \cite{EN-MD0}, or in \cite{TW-Dym} on the ease the space is finite dimensional): 
\begin{cor}[Corollary 5.3 of \cite{Lin-Dym}]
Let $X$ be a compact metrizable space, and let $\sigma: X\to X$ be a minimal homeomorphism. Then the C*-algebra $(\mathrm{C}(X)\rtimes_\sigma\Int)\otimes\mathcal Z$ is classifiable. In particular, if $(X, \sigma)$ has mean dimension zero, the C*-algebra $\mathrm{C}(X)\rtimes_\sigma\Int$ is classifiable.
\end{cor}
\begin{proof}
By Theorem 4.1 of \cite{ENST-ASH}, the C*-algebra $(\mathrm{C}(X)\rtimes_\sigma\Int)\otimes Q$ is locally ASH. By Theorem \ref{classification}, the C*-algebra $(\mathrm{C}(X)\rtimes_\sigma\Int)\otimes Q$ belongs to the class $\mathrm{TA}\mathcal C_0$, and hence the C*-algebra $(\mathrm{C}(X)\rtimes_\sigma\Int)\otimes\mathcal Z$ is classifiable. 

If $(X, \sigma)$ has mean dimension zero, then it follows from \cite{EN-MD0} that $$\mathrm{C}(X)\rtimes_\sigma\Int\cong (\mathrm{C}(X)\rtimes_\sigma\Int)\otimes\mathcal Z,$$
and so the C*-algebra $\mathrm{C}(X)\rtimes_\sigma\Int$ is classifiable.
\end{proof}

\bibliographystyle{plainurl}

\end{document}